\documentclass[12pt]{amsart}

\usepackage[mathscr]{eucal}
\usepackage{amsmath,amssymb,amscd,amsthm}
\usepackage{color}
\usepackage{amsmath,amsthm,amssymb,graphicx}
\usepackage{epsfig}
\newtheorem{theorem}{Theorem}

\newtheorem{definition}{Definition}

\newtheorem{criterion}{Criterion}
\newtheorem{proposition}{Proposition}
\theoremstyle{definition}
\newtheorem{remark}{Remark}
\theoremstyle{plane}

\def \beq{ \begin{equation} }
\def \eeq{\end{equation}}

\usepackage[top=4cm, bottom=4cm, left=3.5cm, right=3.5cm]{geometry}

\title{Rotopulsators of the curved $N$-body problem}
\begin{document}
\maketitle
\markboth{Florin Diacu and Shima Kordlou}{Rotopulsators of the curved $N$-body problem}
\author{\begin{center}
{\bf Florin Diacu}$^{1,2}$ and {\bf Shima Kordlou}$^2$\\
\smallskip
{\footnotesize $^1$Pacific Institute for the Mathematical Sciences\\
and\\
$^2$Department of Mathematics and Statistics\\
University of Victoria\\
P.O.~Box 3060 STN CSC\\
Victoria, BC, Canada, V8W 3R4\\
diacu@uvic.ca and shimak@uvic.ca\\
}\end{center}

}

\vskip0.5cm

\begin{center}
\today
\end{center}

\vspace{2cm}

\begin{abstract}
We consider the $N$-body problem in spaces of constant curvature and study its rotopulsators, i.e.\ solutions for which the configuration of the bodies rotates and changes size during the motion. Rotopulsators fall naturally into five groups: positive elliptic, positive elliptic-elliptic, negative elliptic, negative hyperbolic, and negative elliptic-hyperbolic, depending on the nature and number of their rotations and  on whether they occur in spaces of positive or negative curvature. After obtaining existence criteria for each type of rotopulsator, we derive their conservation laws. We further deal with the existence and uniqueness of some classes of rotopulsators in the 2- and 3-body case and prove two general results about the qualitative behaviour of rotopulsators. More precisely, for positive curvature we show that there is no foliation of the 3-sphere with Clifford tori such that the motion of each body is confined to some Clifford torus. For negative curvature, a similar result is proved relative to foliations of the hyperbolic 3-sphere with hyperbolic cylinders. 
\end{abstract}

\newpage

\tableofcontents

 
\section{Introduction}

The goal of this paper is to study rotopulsators of the curved $N$-body problem, a type of solution that extends the concept of homographic orbit from the Euclidean space to spaces of constant curvature. In curved spaces, similar geometric figures are also congruent, so the word homographic is not suited for describing orbits that rotate and expand or contract during the motion. We will therefore introduce the concept of rotopulsator, which overlaps with the Euclidean notion of homographic solution of the $N$-body problem when the curvature of the space becomes zero. But before describing the results we prove in this paper, let us give the overall motivation of this research and provide a brief history of the $N$-body problem in spaces of constant curvature.


\section{Motivation}

The curved $N$-body problem offers an opportunity to look into the nature of the physical space. How do we measure the shortest distance between two points: along a straight line, an arc of a great circle, or a geodesic of some other manifold? Apparently, Gauss tried to answer this question by measuring the angles of triangles formed by three mountain peaks, to decide whether their sum was smaller or larger than $\pi$ radians. But his experiments failed because no deviation from $\pi$ could be detected outside the unavoidable measurement errors. It thus became clear that if the physical space is not flat, the deviation from zero curvature is very small. Recent cosmological experiments involving the background radiation made physicists believe that space is Euclidean, although they have no proof so far, their results being as inconclusive as those of Gauss. In fact, from the mathematical point of view, zero curvature is highly unlikely, if compared to a continuum of possibilities for positive or negative curvature.

The study of the curved $N$-body problem offers new insight since we can observe celestial motions from Earth. If we prove the existence of orbits that characterize only one of the positive, zero, or negative constant curvature spaces, i.e.\ such orbits don't occur in any of the two other possible spaces, and are stable too, then we can hope to decide the shape of physical space by astronomical observations.

This dynamical approach towards understanding the geometry of the universe may succeed at the local, but not the global, level because celestial motions at large scales are mainly radial: galaxies and clusters of galaxies just move away from each other. Interesting celestial orbits occur only in solar systems. But the study of the curved $N$-body problem offers a new mathematical playfield that can shed some light on the Euclidian case through the study of the bifurcations that occur when the curvature tends to zero, and may lead to a better understanding of several mathematical questions, including those related to the singularities that occur in the motion of particle systems, \cite{Diacu8}, \cite{Diacu1}.


\section{A brief history of the curved $N$-body problem}

In the 1830s, J\'anos Bolyai and Nikolai Lobachevsky independently proposed a 2-body problem in the hyperbolic space $\mathbb H^3$, \cite{Bolyai}, \cite{Lobachevsky}. They suggested the use of a force that is inversely proportional with the area of a sphere of radius $r$, where $r$ is the distance between the bodies. Their line of thought followed that of Gauss, who had viewed gravitation as inversely proportional to the area of an Euclidean sphere
in Euclidean space. But neither Bolyai nor Lobachevsky came up with an analytic expression for this new force. In 1870, Ernest Schering pointed out that, in $\mathbb H^3$, the area of a sphere of radius $r$ is proportional to $\sinh^2r$, so he defined a potential that involves $\coth r$, \cite{Schering}. Wilhelm Killing naturaly extended this idea to the sphere $\mathbb S^3$ with the help of a potential proportional to $\cot r$, \cite{Killing}.

During the first couple of years of the 20th century, Heinrich Liebmann studied the Kepler problem (the motion of a body around a fixed centre) for the cotangent potential and recovered Kepler's laws in slightly modified form, \cite{Liebmann1}. Moreover, he found that a property earlier proved by Joseph Bertrand for the Newtonian potential, namely that all bounded orbits of the Kepler problem are closed, \cite{Bertrand}, was also true for the cotangent potential, \cite{Liebmann2}. These results, together with the fact that both the Newtonian and the cotangent potential of the Kepler problem are harmonic functions in the 3-dimensional case, established the cotangent potential among the problems worth researching in celestial mechanics. More recently, results in this direction were obtained by several Spanish and Russian mathematicians, \cite{Carinena}, \cite{Kozlov}, \cite{Shchepetilov}. It is interesting to note that, unlike in Euclidean space, the curved Kepler problem and the curved 2-body problem are not equivalent. The former is integrable, but the latter is not, \cite{Shchepetilov}, so its study appears to be more complicated than that of the classical case.

Recently, we found a new setting that allowed the generalization of the 2-body case to any number $N\ge 2$ of bodies. We showed in several papers, \cite{Diacu7}, \cite{Diacu3}, \cite{Diacu4}, that the equations of motion can be simultaneously written for positive and negative curvature. The idea was to use the hyperbolic sphere, i.e.\ Weierstrass's model of hyperbolic geometry, given by the upper sheet of the hyperboloid of two sheets embedded in the Minkowski space. By suitable coordinate and time transformations, this study can be reduced to $\mathbb S^3$ and $\mathbb H^3$, or to $\mathbb S^2$ and $\mathbb H^2$ in the 2-dimensional case.
For the latter, the equations in intrinsic coordinates were also obtained, \cite{Perez}, \cite{Diacu6}. So far, we studied the singularities of the equations and of the solutions, the various classes of relative equilibria (i.e.\ orbits whose mutual distances are constant in time), and some rotopulsating orbits in the 2-dimensional case, \cite{Diacu1}, \cite{Diacu2}, \cite{Diacu3}, \cite{Diacu4}, \cite{Diacu5}, \cite{Diacu6}, \cite{Diacu7}, \cite{Diacu8}. This paper provides a first investigation of the rotopulsating orbits in $\mathbb S^3$ and $\mathbb H^3$.

\section{Summary of results}

The remaining part of this paper is structured as follows. We first introduce in Section 5 the equations of motion in $\mathbb S^3$ and $\mathbb H^3$ and obtain their seven integrals of motion: one for the energy and six for the total angular momentum. In Section 6, we define the concept of rotopulsating orbit (or rotopulsator) and classify these solutions into five types, which we call positive elliptic, positive elliptic-elliptic, negative elliptic, negative hyperbolic, and negative elliptic-hyperbolic, depending on the manifold in which they move and on the nature and number of their rotations, which are determined by Lie groups of isometric transformations in $\mathbb S^3$ and $\mathbb H^3$.

The main results are stated and proved in Sections 7 through 13. Sections 7, 8, 10, 11, and 12 have the same structure, each providing a criterion for the existence of the type of rotopulsator it analyzes, giving the corresponding conservation laws, and proving existence and uniqueness results for classes of rotopulsators or relative equilibria for systems of two or three bodies. Thus Section 7 deals with positive elliptic, Section 8 with positive elliptic-elliptic, Section 10 with negative elliptic, Section 11 with negative hyperbolic, and  Section 12 with negative elliptic-hyperbolic rotopulsators and/or relative equilibria. The examples provided in each of these sections are of Lagrangian type (i.e.\ the bodes lie at the vertices of rotating equilateral triangles) for the positive elliptic, positive elliptic-elliptic, and negative elliptic rotopulsators, and of Eulerian type (i.e.\ the bodies lie on the same rotating geodesic) for the negative hyperbolic and negative elliptic-hyperbolic rotopulsators.

Sections 10 and 13 offer a theorem each, whose object is to describe the qualitative behaviour of some classes of rotopulsators in $\mathbb S^3$ and $\mathbb H^3$, respectively. The first theorem shows that, for rotopulsators of $\mathbb S^3$, for any foliation of the sphere with Clifford tori, none of the bodies can stay confined to some Clifford torus, so at least one body passes through a continuum of such surfaces. The second theorem proves a similar result for rotopulsators in $\mathbb H^3$, where the Clifford tori are replaced by hyperbolic cylinders.

An interesting finding is that of a class of Lagrangian relative equilibria (see Subsection \ref{pee-Lagr}) that cannot be generated from any single element of the underlying subgroup $SO(2)\times SO(2)$ of the Lie group $SO(4)$. (From the geometric-dynamical point of view this is very much like viewing the uniform motion of a point around a circle in its projection on some inclined plane. In projection, the motion appears elliptic and nonuniform.)   
Nevertheless, using suitable coordinate transformations, i.e.\ rotations of the frame, it is possible to find coordinates in which the solution can be generated by an element of the underlying torus $SO(2)\times SO(2)$. But the qualitative behaviour of the system can also be studied, without loss of information, in the original coordinates.


\section{Equations of motion}

Consider $N$ bodies (point masses, point particles) of masses $m_1,\dots, m_N>0$ moving in 
$\mathbb{S}^3$ (thought as embedded in the Euclidean space $\mathbb R^4$) or $\mathbb H^3$ (embedded in the Minkowski space $\mathbb R^{3,1}$),
where
\begin{equation*}
\begin{split}
{\mathbb S}^3=\{(w,x,y,z)\ | \ w^2+x^2+y^2+z^2=1\},\hspace{0.9cm}\\
{\mathbb H}^3=\{(w,x,y,z)\ | \ w^2+x^2+y^2-z^2=-1, \ z>0\}.
\end{split}
\end{equation*}
In previous work, we derived the equations of motion of the curved $N$-body problem using constrained Lagrangian dynamics and showed that, by suitable coordinate and time-rescaling transformations, the study of the problem can be reduced to $\mathbb S^3$, for positive curvature, and to $\mathbb H^3$, for negative curvature, as long as we deal only with qualitative properties, \cite{Diacu3}, \cite{Diacu4}.

The configuration of the system is described by the $4N$-dimensional vector
$$
{\bf q}=({\bf q}_1,\dots,{\bf q}_N),
$$
where ${\bf q}_i=(w_i,x_i,y_i,z_i), i=1,\dots, N$, denote the position vectors of the bodies.  The equations of motion are given by the second-order system
\begin{equation}
\label{both}
\ddot{\bf q}_i=\sum_{j=1,j\ne i}^N\frac{m_j[{\bf q}_j-\sigma({\bf q}_i\cdot {\bf q}_j){\bf q}_i]}{[\sigma-\sigma({\bf q}_i\cdot {\bf q}_j)^2]^{3/2}}-\sigma(\dot{\bf q}_i\cdot \dot{\bf q}_i){\bf q}_i, \ \ i=1,\dots, N,
\end{equation}
with initial-condition constraints
\begin{equation}
{\bf q}_i(0)\cdot{\bf q_i}(0)=\sigma, \ \ {\bf q}_i(0)\cdot\dot{\bf q}_i(0)=0, \ \ i=1,\dots, N.
\end{equation}
Here $\cdot$ is the standard inner product of signature $(+,+,+,+)$ in $\mathbb S^3$, but the Lorentz inner product of signature $(+,+,+,-)$ in $\mathbb H^3$, and
$$
\sigma=
\begin{cases}
+1,\ \ {\rm in}\ \ {\mathbb S^3},\cr
-1,\ \ {\rm in}\ \ {\mathbb H^3},
\end{cases}
$$
denotes the signum function. Since the equations of motion and the constraints on the initial conditions imply that 
$$
{\bf q}_i\cdot{\bf q_i}=\sigma, \ \ {\bf q}_i\cdot\dot{\bf q}_i=0, \ \ i=1,\dots, N,
$$
for all time, equations \eqref{both} can be viewed as a $6N$-dimensional first-order system of ordinary differential equations. The gravitational force acting on each body has an attractive component (the above sum) and a term (involving the velocities) that corresponds to the constraints.

As a consequence of Noether's theorem, system \eqref{both} has the scalar integral of energy,
$$
T({\bf q},\dot{\bf q})-U({\bf q})=h,
$$
where 
$$U({\bf q})=\sum_{1\le i<j\le N}\frac{\sigma m_im_j{\bf q}_i\cdot{\bf q}_j}{[\sigma-\sigma({\bf q}_i\cdot{\bf q}_j)^2]^{3/2}}$$ 
is the force function ($-U$ representing the potential), which stems from the cotangent of the distance, and
$$
T({\bf q},\dot{\bf q})=\frac{1}{2}\sum_{i=1}^Nm_i(\dot{\bf q}_i\cdot\dot{\bf q}_i)
(\sigma{\bf q}_i\cdot{\bf q}_i)
$$
is the kinetic energy, with $h$ representing an integration constant. System \eqref{both} also has the 6-dimensional integral of the total angular momentum,
$$
\sum_{i=1}^Nm_i{\bf q}_i\wedge\dot{\bf q}_i={\bf c},
$$
where $\wedge$ is the wedge product and ${\bf c}=(c_{wx},c_{wy},c_{wz},c_{xy},c_{xz},c_{yz})$ denotes an integration vector, each component measuring the rotation of the system about the origin of the frame relative to the plane corresponding to the bottom indices. On components, the 6 integrals are given by the equations
\begin{align*}\label{angularmomentum}
\sum_{i=1}^Nm_i(w_i\dot{x}_i-\dot{w}_ix_i)&=c_{wx}, & \sum_{i=1}^Nm_i(w_i\dot{y}_i-\dot{w}_iy_i)&=c_{wy},\\
\sum_{i=1}^Nm_i(w_i\dot{z}_i-\dot{w}_iz_i)&=c_{wz},& \sum_{i=1}^Nm_i(x_i\dot{y}_i-\dot{x}_iy_i)&=c_{xy},\\
\sum_{i=1}^Nm_i(x_i\dot{z}_i-\dot{x}_iz_i)&=c_{xz},&  \sum_{i=1}^Nm_i(y_i\dot{z}_i-\dot{y}_iz_i)&=c_{yz}.
\end{align*}

Using the notations
$$
q_{ij}:={\bf q}_i\cdot{\bf q}_j\ \ {\rm and} \ \ \dot q_{ii}:=\dot{\bf q}_i\cdot\dot{\bf q}_i,
$$
we can explicitly write the equations of motion in $\mathbb S^3$ as
\begin{equation}
\label{positive}
\ddot{\bf q}_i=\sum_{\stackrel{j=1}{j\ne i}}^N\frac{m_j({\bf q}_j-q_{ij}{\bf q}_i)}{(1-q_{ij}^2)^{3/2}}-\dot q_{ii}{\bf q}_i,\ \
q_{ii}=1, \ \ i=1,2,\dots,N,
\end{equation}
and in $\mathbb H^3$ as
\begin{equation}
\label{negative}
\ddot{\bf q}_i=\sum_{\stackrel{j=1}{j\ne i}}^N\frac{m_j({\bf q}_j+q_{ij}{\bf q}_i)}{(q_{ij}^2-1)^{3/2}}+\dot q_{ii}{\bf q}_i,\ \
q_{ii}=-1, \ \ i=1,2,\dots,N.
\end{equation}
It is important to recall that the inner product that occurs in the expressions of $q_{ij}$ and $\dot q_{ii}$ is not the same in the above two systems. From now on we will refer to equations \eqref{positive} when we study motions in $\mathbb S^3$ and to equations \eqref{negative} when dealing with the dynamics of the bodies in $\mathbb H^3$.


\section{Basic definitions}

In this section we define several types of rotopulsating orbits of the curved $N$-body problem, a classification that follows naturally from the isometry groups of $\mathbb S^3$ and $\mathbb H^3$. The rotopulsating orbit extends the concept of Euclidean homographic solution to spaces of nonzero constant curvature. In two previous papers we introduced this concept in the 2-dimensional case and kept using the name ``homographic'' for it, \cite{Diacu2}, \cite{Diacu5}. Our idea was that the configurations we studied (mostly polygons) remained homographic if viewed in the ambient Euclidean space. But it seems more natural to regard configurations in intrinsic terms, the more so when we move from two to three dimensions. 

In $\mathbb S^2, \mathbb S^3, \mathbb H^2$, and $\mathbb H^3$, however, the concept of similarity, which corresponds to the adjective homographic, makes little sense, since, for instance, the only similar triangles are the congruent ones. So to extend the concept of homographic orbit to spaces of constant curvature, the terminology needs, on one hand, to capture somehow the expansion/contraction aspect as well as the rotational component of the motion, and, on the other hand, to agree with the properties described by the original definition when the curvature tends to zero. We therefore introduce here a new adjective, {\it rotopulsating}, which preserves the features of the orbit without implying similarity of the configuration. For simplicity, rotopulsating orbits will also be called {\it rotopulsators}.

The definitions we provide below follow naturally from the concept of relative equilibrium of the curved $N$-body problem, defined in \cite{Diacu3}, \cite{Diacu4}. We introduced there various kinds of relative equilibria in terms of the isometric rotation groups of $\mathbb S^3$ and $\mathbb H^3$. The rotopulsators differ from relative equilibria by having nonuniform rotations and nonconstant mutual distances, as we will further see. To reconcile the two concepts, we also offer a new definition for relative equilibria.

\begin{definition}[{\bf Positive elliptic rotopulsators and relative equilibria}]
\label{def-positive-elliptic}
A solution of system \eqref{positive} in $\mathbb S^3$ is called a positive elliptic rotopulsator if it is of the form
\begin{equation}\label{positive-elliptic}
\begin{split}
{\bf q}=({\bf q}_1, {\bf q}_2,\dots,{\bf q}_N), \  \ {\bf q}_i=
(w_i,x_i,y_i,z_i),\ \ i=1,2,\dots, N,\hspace{1.3cm}\\
w_i=r_i(t)\cos[\alpha(t)+a_i],\ \ x_i=r_i(t)\sin[\alpha(t)+a_i],\ \ y_i=y_i(t),\ \ z_i=z_i(t),
\end{split}
\end{equation}
where $a_i,\ i=1,2,\dots, N$, are constants, $\alpha$ is a nonconstant function, $r_i, y_i$, and $z_i$ satisfy the conditions  
\begin{equation}\label{conditions-pe}
0\le r_i\le 1, \ \ \ -1\le y_i, z_i\le 1,\ \ \ {\rm and}\ \ \
r_i^2+y_i^2+z_i^2=1,\ \ \ i=1,2,\dots, N,
\end{equation}
and there are at least two indices $i,j\in\{1,2,\dots,N\}, i\ne j$, such that $q_{ij}$ is not constant.  If the quantities $q_{ij}$ are constant for all $i,j\in\{1,2,\dots,N\}, i\ne j$, then the solution is called a positive elliptic relative equilibrium.
\end{definition}

\begin{remark}
The condition that $\alpha$ is nonconstant is imposed to ensure  that the system has an elliptic rotation relative to the $wx$-plane; the fact that it has no rotation relative to the $yz$-plane follows from \eqref{conditions-pe} and the corresponding integral of the angular momentum (see also Remark \ref{def-nonstandard} below). Rotations relative to other base planes may occur.
\end{remark}

\begin{definition}[{\bf Positive elliptic-elliptic rotopulsators and relative equilibria}]
\label{def-positive-elliptic-elliptic}
A solution of system \eqref{positive} in $\mathbb S^3$  is called a  positive elliptic-elliptic rotopulsator if it is of the form
\begin{equation}\label{positive-elliptic-elliptic}
\begin{split}
{\bf q}=({\bf q}_1, {\bf q}_2,\dots,{\bf q}_N), \ {\bf q}_i=
(w_i,x_i,y_i,z_i),\ i=1,2,\dots, N,\hspace{-0.7cm}\\
w_i=r_i(t)\cos[\alpha(t)+a_i],\ x_i=r_i(t)\sin[\alpha(t)+a_i],\\
y_i=\rho_i(t)\cos[\beta(t)+b_i],\ z_i=\rho_i(t)\sin[\beta(t)+b_i],
\end{split}
\end{equation}
where $a_i, b_i,\ i=1,2,\dots, N$, are constants, $\alpha$ and $\beta$ are nonconstant functions, and $r_i$, $\rho_i$ satisfy the conditions  
$$
0\le r_i,\rho_i\le 1\ \ {\rm and}\ \
r_i^2+\rho_i^2=1,\ i=1,2,\dots, N,
$$
and there are at least two indices $i,j\in\{1,2,\dots,N\}, i\ne j$, such that $q_{ij}$ is not constant.  If the quantities $q_{ij}$ are constant for all $i,j\in\{1,2,\dots,N\}, i\ne j$, then the solution is called a positive elliptic-elliptic relative equilibrium.
\end{definition}

\begin{remark}
The conditions that $\alpha$ and $\beta$ are nonconstant are imposed to ensure that the system has two elliptic rotations, one relative to the $wx$-plane and the other relative to the $yz$-plane. Rotations relative to the other base planes may occur.
\end{remark}

\begin{definition}[{\bf Negative elliptic rotopulsators and relative equilibria}]
\label{def-negative-elliptic}
A solution of system \eqref{negative} in $\mathbb H^3$ is called a negative elliptic  rotopulsator if it is of the form
\begin{equation}
\label{negative-elliptic}
\begin{split}
{\bf q}=({\bf q}_1, {\bf q}_2,\dots,{\bf q}_N), \  \ {\bf q}_i=
(w_i,x_i,y_i,z_i),\ \ i=1,2,\dots, N,\hspace{1.3cm}\\
w_i=r_i(t)\cos[\alpha(t)+a_i],\ \ x_i=r_i(t)\sin[\alpha(t)+a_i],\ \ y_i=y_i(t),\ \ z_i=z_i(t),
\end{split}
\end{equation}
where $a_i,\ i=1,2,\dots, N$, are constants, $\alpha$ is a nonconstant function, $r_i, y_i$, and $z_i$ satisfy the conditions  
\begin{equation}\label{conditions-ne}
z_i\ge 1\ \ \ {\rm and}\ \ \
r_i^2+y_i^2-z_i^2=-1,\ \ \ i=1,2,\dots, N,
\end{equation}
and there are at least two indices $i,j\in\{1,2,\dots,N\}, i\ne j$, such that $q_{ij}$ is not constant.  If the quantities $q_{ij}$ are constant for all $i,j\in\{1,2,\dots,N\}, i\ne j$, then the solution is called a negative elliptic relative equilibrium.
\end{definition}

\begin{remark}
The condition that $\alpha$ is nonconstant is imposed to ensure that the system has an elliptic rotation relative to the $wx$-plane; the fact that it has no hyperbolic rotation relative to the $yz$-plane follows from \eqref{conditions-ne} and the corresponding integral of the angular momentum (see also Remark \ref{def-nonstandard} below). Rotations relative to other base planes may occur.
\end{remark}

\begin{definition}[{\bf Negative hyperbolic rotopulsators and relative equilibria}]
\label{def-negative-hyperbolic}
A solution of system \eqref{negative} in $\mathbb H^3$ is called a negative hyperbolic rotopulsator if it is of the form
\begin{equation}
\label{negative-hyperbolic}
\begin{split}
{\bf q}=({\bf q}_1, {\bf q}_2,\dots,{\bf q}_N), \  \ {\bf q}_i=
(w_i,x_i,y_i,z_i),\ \ i=1,2,\dots, N,\hspace{1.3cm}\\
w_i=w_i(t),\ x_i=x_i(t),\ y_i=\rho_i(t)\sinh[\beta(t)+b_i], \ z_i=\rho_i(t)\cosh[\beta(t)+b_i],
\end{split}
\end{equation}
where $b_i,\ i=1,2,\dots, N$, are constants, $\beta$ is a nonconstant function, $w_i, x_i, z_i$, and $\rho_i$ satisfy the conditions  
\begin{equation}\label{conditions-nh}
z_i\ge 1\ \ \ {\rm and}\ \ \
w_i^2+x_i^2-\rho_i^2=-1,\ \ \ i=1,2,\dots, N,
\end{equation}
and there are at least two indices $i,j\in\{1,2,\dots,N\}, i\ne j$, such that $q_{ij}$ is not constant.  If the quantities $q_{ij}$ are constant for all $i,j\in\{1,2,\dots,N\}, i\ne j$, then the solution is called a negative hyperbolic relative equilibrium.
\end{definition}

\begin{remark}
The condition that $\beta$ is nonconstant is imposed to ensure that the system has a hyperbolic rotation relative to the $yz$-plane; the fact that it has no elliptic rotation relative to the $wx$-plane follows from \eqref{conditions-nh} and the corresponding integral of the angular momentum (see also Remark \ref{def-nonstandard} below). Rotations relative to other base planes may occur.
\end{remark}

\begin{definition}[{\bf Negative elliptic-hyperbolic rotopulsators and relative equilibria}]
\label{def-negative-elliptic-hyperbolic}
A solution of system \eqref{negative} in $\mathbb H^3$ is called a negative elliptic-hyperbolic rotopulsator if it is of the form
\begin{equation}
\label{negative-elliptic-hyperbolic}
\begin{split}
{\bf q}=({\bf q}_1, {\bf q}_2,\dots,{\bf q}_N), \  \ {\bf q}_i=
(w_i,x_i,y_i,z_i),\ \ i=1,2,\dots, N,\hspace{0.2cm}\\
w_i=r_i(t)\cos[\alpha(t)+a_i],\ \ x_i=r_i(t)\sin[\alpha(t)+a_i],\hspace{1.1cm}\\\  
y_i=\rho_i(t)\sinh[\beta(t)+b_i], \ z_i=\rho_i(t)\cosh[\beta(t)+b_i],\hspace{0.85cm}\
\end{split}
\end{equation}
where $a_i, b_i,\ i=1,2,\dots, N$, are constants,  $\alpha$ and $\beta$ are nonconstant functions, $r_i,\eta_i$, $z_i$ satisfy the conditions  
$$
z_i\ge 1\ \ \ {\rm and}\ \ \
r_i^2-\rho_i^2=-1,\ \ \ i=1,2,\dots, N,
$$
and there are at least two indices $i,j\in\{1,2,\dots,N\}, i\ne j$, such that $q_{ij}$ is not constant.  If the quantities $q_{ij}$ are constant for all $i,j\in\{1,2,\dots,N\}, i\ne j$, then the solution is called a negative elliptic-hyperbolic relative equilibrium.
\end{definition}

\begin{remark}
The conditions that $\alpha$ and $\beta$ are not constant are imposed to ensure that the system has an elliptic rotation relative to the $wx$-plane and a hyperbolic rotation relative to the $yz$-plane. Rotations relative to other base planes may occur.
\end{remark}

\begin{remark}
Notice that we ignored a class of isometries in $\mathbb H^3$: the parabolic rotations, and did not provide a definition for solutions of this type. The reason for omitting this case is that, as proved in \cite{Diacu3} and \cite{Diacu4}, relative equilibria that stem from parabolic rotations do not exist, and it is easy to show using the same idea that rotopulsators of parabolic type do not exist either. 
\end{remark}

\begin{remark}\label{def-nonstandard}
There are alternative ways to define the above classes of rotopulsating solutions of the curved $N$-body problem. Indeed, positive elliptic rotopulsators could be considered as positive elliptic-elliptic rotopulsators with $\beta\equiv 0$.
Then
$$
y_i(t)=\rho_i(t)\cos b_i\ \ {\rm and}\ \ z_i(t)=\rho_i(t)\sin b_i, \ i=1,2,\dots,N,
$$
with $a_i, b_i$ constants, which makes sense, given that $y_i^2+z_i^2=\rho_i^2, \ i=1,2,\dots, N.$
Similarly, negative elliptic and negative hyperbolic rotopulsators could be defined as negative elliptic-hyperbolic rotopulsators with $\beta\equiv 0$ and $\alpha\equiv 0$, respectively. But our choice of five distinct definitions is more convenient for computations and will help us emphasize, unambiguously, certain properties specific to each of these solutions. Nevertheless, we will use this remark later in the proofs of Theorems 1 and 2.
\end{remark}

\begin{remark}
In \cite{Diacu5} we defined rotopulsators of the 3-body problem (called homographic orbits there for reasons we invoked earlier) in $\mathbb S^2$ and $\mathbb H^2$ in a narrower sense by asking that the Euclidean plane formed by the 3 bodies is all the time parallel with the $xy$-plane. In Definition \ref{def-positive-elliptic} reduced to $\mathbb S^2$, for instance, this implies that the condition for a relative equilibrium (which means that the mutual distances between the bodies remain constant during the motion) is equivalent to saying that $r_i$ is constant for all $i=1,2,\dots,N$. But, interesting enough, in Definition \ref{def-positive-elliptic-elliptic} the functions $r_i$ (and consequently $\rho_i$) may vary in spite of the fact that the quantities $q_{ij}$ stay constant, i.e.\ the mutual distances don't vary in time. The reason for this behaviour is that the corresponding relative equilibrium cannot be generated from a single element of the natural subgroup $SO(2)\times SO(2)$
of the Lie group $SO(4)$ that arises from the $wx$ and $yz$ coordinate pairs of the chosen reference frame. Nevertheless, a classical result, which claims that in a semisimple compact Lie group every element is contained in a maximal torus, \cite{Montaldi}, shows that a suitable change of coordinates leads to a reference frame in which all functions $r_i$ are constant. In particular, $SO(4)$ is a semisimple compact Lie group, so it satisfies the above result. Moreover, its maximal tori are the subgroups $SO(2)\times SO(2)$. But since it is impossible to know a priori which reference system to choose in order to make the functions $r_i$ constant, it is easier to define relative equilibria by asking that the functions $q_{ij}$, and not the functions $r_i$, are constant, as we did in all the above definitions. 
\end{remark}


\section{Positive elliptic rotopulsators}

In this section we analyze the solutions given in Definition \ref{def-positive-elliptic}. We first introduce a criterion for finding them, then obtain the conservation laws, and finally prove the existence of a particular class of orbits, namely the positive elliptic Lagrangian rotopulsators of the 3-body problem in $\mathbb S^3$.

\subsection{Criterion for positive elliptic rotopulsators or relative equilibria}

The following result provides necessary and sufficient conditions for the existence of positive elliptic rotopulsators or relative equilibria in $\mathbb S^3$.

\begin{criterion}
\label{pe-existence}
A solution candidate of the form \eqref{positive-elliptic} is a positive elliptic rotopulsator for system \eqref{positive} if and only if
\begin{equation}
\label{al-pe}
\dot\alpha=\frac{c}{\sum_{j=1}^Nm_jr_j^2}, 
\end{equation}
where $c\ne 0$ is a constant, there are at least two distinct indices $i,j\in\{1,2,\dots,N\}$ such that $q_{ij}$ is not constant, and the variables $y_i, z_i,r_i, \ i=1,2,\dots, N,$ satisfy the first-order system of\ \! $5N$ equations (with\ \! $N$ constraints: $r_i^2+y_i^2+z_i^2=1$, $i=1,2,\dots,N$),
\begin{equation}
\label{sys-crit-pe}
\begin{cases}
\dot y_i=u_i\cr
\dot z_i=v_i\cr
\dot{u}_i=\sum_{\stackrel{j=1}{j\ne i}}^N\frac{m_j(y_j-q_{ij}y_i)}{(1-q_{ij}^2)^{3/2}}-F_i({\bf y}, {\bf z}, u_i, v_i)y_i\cr
\dot{v}_i=\sum_{\stackrel{j=1}{j\ne i}}^N\frac{m_j(z_j-q_{ij}z_i)}{(1-q_{ij}^2)^{3/2}}-F_i({\bf y}, {\bf z}, u_i, v_i)z_i\cr
r_i\ddot\alpha+2\dot{r}_i\dot\alpha=\sum_{\stackrel{j=1}{j\ne i}}^N
\frac{m_jr_j\sin(a_j-a_i)}{(1-q_{ij}^2)^{3/2}},
\end{cases}
\end{equation}
where ${\bf y}=(y_1,y_2,\dots,y_N), {\bf z}=(z_1,z_2,\dots, z_N)$,
\begin{equation}
F_i({\bf y}, {\bf z}, u_i, v_i):=\frac{u_i^2+v_i^2-(y_iv_i-z_iu_i)^2}{1-y_i^2-z_i^2}+
\frac{c^2(1-y_i^2-z_i^2)}{\big[\sum_{j=1}^Nm_j(1-y_j^2-z_j^2)\big]^2},
\end{equation}
$i=1,2,\dots, N,$ and, for any $i,j\in\{1,2,\dots,N\}$,
$$
q_{ij}=r_ir_j\cos(a_i-a_j)+y_iy_j+z_iz_j.
$$
If the quantities $q_{ij}$ are constant for all distinct indices $i,j\in\{1,2,\dots,N\}$, then the solution is a relative equilibrium. If $q_{ij}=\pm 1$ for some distinct $i,j\in\{1,2,\dots,N\}$, then such solutions don't exist.
\end{criterion}
\begin{proof}
Consider a solution candidate of the form \eqref{positive-elliptic}. A straightforward computation shows that, for any $i,j\in\{1,2,\dots,N\}$, $q_{ij}$ is of the form given in the above statement. Moreover, for any $i=1,2,\dots, N$, we find that 
$$
\dot q_{ij}=\frac{\dot{y}_i^2+\dot{z}_i^2-(y_i\dot{z}_i-z_i\dot{y}_i)^2+(1-y_i^2-z_i^2)^2\dot\alpha^2}{1-y_i^2-z_i^2}.
$$
For all $\ i=1,2,\dots, N$, each $r_i$ can be expressed in terms of $y_i$ and $z_i$ to obtain
$$
r_i=(1-y_i^2-z_i^2)^{\frac{1}{2}},\ \ \
\dot{r}_i=-\frac{y_i\dot{y}_i+z_i\dot{z}_i}{(1-y_i^2-z_i^2)^{\frac{1}{2}}},
$$
$$
\ddot{r}_i=\frac{(y_i\dot{z}_i-z_i\dot{y}_i)^2-\dot{y}_i^2-\dot{z}_i^2-(1-y_i^2-z_i^2)(y_i\ddot{y}_i+z_i\ddot{z}_i)}{(1-y_i^2-z_i^2)^{\frac{3}{2}}}.
$$
Substituting a candidate solution of the form \eqref{positive-elliptic} into system \eqref{positive} and employing the above formulas, we obtain for the equations corresponding to $\ddot{y}_i$ and $\ddot{z}_i$ that
\begin{equation}
\label{yi}
\ddot{y}_i=\sum_{\stackrel{j=1}{j\ne i}}^N\frac{m_j(y_j-q_{ij}y_i)}{(1-q_{ij}^2)^{\frac{3}{2}}}-\frac{[\dot{y}_i^2+\dot{z}_i^2-(y_i\dot{z}_i-z_i\dot{y}_i)^2]y_i}{1-y_i^2-z_i^2}-(1-y_i^2-z_i^2)y_i\dot\alpha^2,
\end{equation}
\begin{equation}
\label{zi}
\ddot{z}_i=\sum_{\stackrel{j=1}{j\ne i}}^N\frac{m_j(z_j-q_{ij}z_i)}{(1-q_{ij}^2)^{\frac{3}{2}}}-\frac{[\dot{y}_i^2+\dot{z}_i^2-(y_i\dot{z}_i-z_i\dot{y}_i)^2]z_i}{1-y_i^2-z_i^2}-(1-y_i^2-z_i^2)z_i\dot\alpha^2.
\end{equation}
For the equations corresponding to $\ddot{w}_i$ and $\ddot{x}_i$, after some long computations that also use \eqref{yi} and \eqref{zi}, we are led to the equations
\begin{equation}
\label{alpha-pe}
r_i\ddot\alpha+2\dot{r}_i\dot\alpha=\sum_{\stackrel{j=1}{j\ne i}}^N
\frac{m_jr_j\sin(a_j-a_i)}{(1-q_{ij}^2)^{\frac{3}{2}}}, \ \ i=1,2,\dots, N.
\end{equation} 
We will further show that equations \eqref{yi}, \eqref{zi}, and \eqref{alpha-pe} lead to the system \eqref{sys-crit-pe}. For this purpose, we first compute $\dot\alpha$.

For every $i=1,2,\dots, N$, multiply the $i$th equation in \eqref{alpha-pe} by $m_ir_i$, add the resulting $N$ equations, and notice that
$$
\sum_{i=1}^N\sum_{\stackrel{j=1}{j\ne i}}^N\frac{m_im_jr_ir_j\sin(a_j-a_i)}{(1-q_{ij}^2)^{\frac{3}{2}}}=0.
$$
Thus we obtain the equation
$$
\ddot\alpha\sum_{i=1}^Nm_ir_i^2+
2\dot\alpha\sum_{i=1}^Nm_ir_i\dot{r}_i=0,
$$
which has the solution
$$
\dot\alpha=\frac{c}{\sum_{i=1}^Nm_ir_i^2}=\frac{c}{\sum_{i=1}^Nm_i(1-y_i^2-z_i^2)},
$$
where $c$ is an integration constant. Consequently,
equations \eqref{yi} and \eqref{zi} become
\begin{equation}
\label{yi-new}
\ddot{y}_i=\sum_{\stackrel{j=1}{j\ne i}}^N\frac{m_j(y_j-q_{ij}y_i)}{(1-q_{ij}^2)^{\frac{3}{2}}}-\frac{[\dot{y}_i^2+\dot{z}_i^2-(y_i\dot{z}_i-z_i\dot{y}_i)^2]y_i}{1-y_i^2-z_i^2}-
\frac{c^2(1-y_i^2-z_i^2)y_i}{\big[\sum_{j=1}^Nm_j(1-y_j^2-z_j^2)\big]^2},
\end{equation}
\begin{equation}
\label{zi-new}
\ddot{z}_i=\sum_{\stackrel{j=1}{j\ne i}}^N\frac{m_j(z_j-q_{ij}z_i)}{(1-q_{ij}^2)^{\frac{3}{2}}}-\frac{[\dot{y}_i^2+\dot{z}_i^2-(y_i\dot{z}_i-z_i\dot{y}_i)^2]z_i}{1-y_i^2-z_i^2}-\frac{c^2(1-y_i^2-z_i^2)z_i}{\big[\sum_{j=1}^Nm_j(1-y_j^2-z_j^2)\big]^2}, 
\end{equation}
$i=1,2,\dots, N$. These equations are equivalent to the first $4N$ equations that appear in \eqref{sys-crit-pe}. 

We still need to show that \eqref{alpha-pe} describes $N$ first-order equations in the unknown functions $r_1, r_2,
\dots, r_N$. A straightforward computation shows that they can be written as
\begin{equation}
\begin{cases}
e_1\dot r_1+b_{12}\dot r_2+b_{13}\dot r_3+\dots+b_{1N}\dot r_N=c_1\cr
b_{21}\dot r_1+e_2\dot r_2+b_{23}\dot r_3+\dots+b_{2N}\dot r_N=c_2\cr
\vdots\cr
b_{N1}\dot r_1+b_{N2}\dot r_2+b_{N3}\dot r_3+\dots+e_N\dot r_N=c_N,
\end{cases}
\end{equation}
where 
$$
e_i=1-m_ir_i^2, \ \ b_{ij}=-m_jr_jr_i,\ \ c_i=\frac{\sum_{j=1}^Nm_jr_j^2}{2c}\sum_{\stackrel{j=1}{j\ne i}}^N\frac{m_jr_j\sin(a_j-a_i)}{1-q_{ij}^2},
$$
$i=1,2,\dots,N$, which is a first-order subsystem of $N$ equations. (In general, this system can be simplified by solving the algebraic system in unknowns $\dot r_1,\dot r_2,\dots, \dot r_N$ leads.) 

The part of the criterion related to relative equilibria follows directly from Definition \ref{def-positive-elliptic}. 
The nonexistence of such solutions if some $q_{ij}=\pm 1$ follows from the fact that at least a denominator cancels in the equations of motion. This remark completes the proof. 
\end{proof}

\subsection{Conservation laws for positive elliptic rotopulsating orbits}

In addition to Criterion \ref{pe-existence}, we would also like to obtain the conservation laws specific to positive elliptic rotopulsators. They follow by straightforward computations using the above proof, the integral of energy, and the six integrals of the total angular momentum.

\begin{proposition}\label{integrals-pe} 
If system \eqref{positive} has a solution of the form \eqref{positive-elliptic}, then the following expressions are constant:

--- energy,
\begin{equation}
\label{energy-pe}
\begin{split}
h=\sum_{i=1}^N\frac{m_i[\dot{y}_i^2+\dot{z}_i^2-(y_i\dot{z}_i-z_i\dot{y}_i)^2]}{2(1-y_i^2-z_i^2)}\hspace{1cm}\\
+\frac{c^2}{2\sum_{i=1}^Nm_i(1-y_i^2-z_i^2)}
-\sum_{1\le i<j\le N}\frac{m_im_jq_{ij}}{(1-q_{ij}^2)^{\frac{1}{2}}};
\end{split}
\end{equation}

--- total angular momentum relative to the $wx$-plane,
\begin{equation}
\label{angmom-wx}
c_{wx}=c,
\end{equation}
where $c\ne 0$ is the constant in the expression \eqref{al-pe} of $\dot\alpha$;

--- total angular momentum relative to the $wy$-plane:
\begin{equation}
\label{angmom-wy}
\begin{split}
c_{wy}=\sum_{i=1}^Nm_i\bigg[(1-y_i^2-z_i^2)^{\frac{1}{2}}\dot{y}_i+\frac{(y_i\dot{y}_i+z_i\dot{z}_i)y_i}{(1-y_i^2-z_i^2)^{\frac{1}{2}}}\bigg]\cos(\alpha+a_i)\\
+\frac{c}{\sum_{i=1}^Nm_i(1-y_i^2-z_i^2)}\sum_{i=1}^Nm_i(1-y_i^2-z_i^2)^{\frac{1}{2}}y_i\sin(\alpha+a_i);
\end{split}
\end{equation}

--- total angular momentum relative to the $wz$-plane,
\begin{equation}
\label{angmom-wz}
\begin{split}
c_{wz}=\sum_{i=1}^Nm_i\bigg[(1-y_i^2-z_i^2)^{\frac{1}{2}}\dot{z}_i+\frac{(y_i\dot{y}_i+z_i\dot{z}_i)z_i}{(1-y_i^2-z_i^2)^{\frac{1}{2}}}\bigg]\cos(\alpha+a_i)\\
+\frac{c}{\sum_{j=1}^Nm_j(1-y_j^2-z_j^2)}\sum_{i=1}^Nm_i(1-y_i^2-z_i^2)^{\frac{1}{2}}z_i\sin(\alpha+a_i);
\end{split}
\end{equation}

--- total angular momentum relative to the $xy$-plane,
\begin{equation}
\label{angmom-xy}
\begin{split}
c_{xy}=\sum_{i=1}^Nm_i\bigg[(1-y_i^2-z_i^2)^{\frac{1}{2}}\dot{y}_i+\frac{(y_i\dot{y}_i+z_i\dot{z}_i)y_i}{(1-y_i^2-z_i^2)^{\frac{1}{2}}}\bigg]\sin(\alpha+a_i)\\
-\frac{c}{\sum_{j=1}^Nm_j(1-y_j^2-z_j^2)}\sum_{i=1}^Nm_i(1-y_i^2-z_i^2)^{\frac{1}{2}}y_i\cos(\alpha+a_i);
\end{split}
\end{equation}

--- total angular momentum relative to the $xz$-plane:
\begin{equation}
\label{angmom-xz}
\begin{split}
c_{xz}=\sum_{i=1}^Nm_i\bigg[(1-y_i^2-z_i^2)^{\frac{1}{2}}\dot{z}_i+\frac{(y_i\dot{y}_i+z_i\dot{z}_i)z_i}{(1-y_i^2-z_i^2)^{\frac{1}{2}}}\bigg]\sin(\alpha+a_i)\\
-\frac{c}{\sum_{j=1}^Nm_j(1-y_j^2-z_j^2)}\sum_{i=1}^Nm_i(1-y_i^2-z_i^2)^{\frac{1}{2}}z_i\cos(\alpha+a_i);
\end{split}
\end{equation}

--- total angular momentum relative to the $yz$-plane:
\begin{equation}
\label{angmom-yz}
c_{yz}=0.
\end{equation}
\end{proposition}

The above result could be also used to prove the nonexistence of some candidates for positive elliptic rotopulsators, by showing that at least one of the above conservation laws is violated.

\subsection{Positive elliptic Lagrangian rotopulsators}

We further provide a class of specific examples of positive elliptic rotopulsators of the curved 3-body problem, namely Lagrangian orbits, i.e.\ bodies that lie at the vertices of a rotating equilateral triangle in $\mathbb S^3$ that changes size, which means that it fails to be similar to itself but has congruent sides at every instant in time. As we will see, these systems rotate relative to the plane $wx$, but have no rotations with respect to the other base planes.

Consider three equal masses, $m_1=m_2=m_3=:m>0$, and a candidate solution of the form
\begin{equation}\label{equiltr-pe}
\begin{split}
{\bf q}=({\bf q}_1,{\bf q}_2,{\bf q}_3), \ \ {\bf q}_i=(w_i,x_i,y_i,z_i), \ \ i=1,2,3,\hspace{1.cm}\\
w_1=r(t)\cos\alpha(t),\ x_1=r(t)\sin\alpha(t),\
y_1=y(t),\ z_1=z(t),\hspace{0.1cm}\\
w_2=r(t)\cos[\alpha(t)+2\pi/3],\ x_2=r(t)\sin[\alpha(t)+2\pi/3],\hspace{0.6cm}\\\
y_2=y(t),\ z_2=z(t),\hspace{3.4cm}\\\
w_3=r(t)\cos[\alpha(t)+4\pi/3],\ x_3=r(t)\sin[\alpha(t)+4\pi/3],\hspace{0.6cm}\\
y_3=y(t),\ z_3=z(t).\hspace{3.4cm}
\end{split}
\end{equation}
With the help of Criterion \ref{pe-existence}, we will further show that these are indeed solutions of curved 3-body problem in $\mathbb S^3$. 

\begin{proposition}
\label{prop-pe-L}
Consider the curved $3$-body problem in $\mathbb S^3$ given by system \eqref{positive} with $N=3$. Then, except for a negligible set of orbits formed by positive elliptic Lagrangian relative equilibria, every candidate solution of the form \eqref{equiltr-pe} is a positive elliptic Lagrangian rotopulsator, which rotates relative to the plane $wx$, but has no rotation with respect to the planes $wy, wz, xy, xz,$ and $yz$. 
\end{proposition}
\begin{proof}
Let us consider a candidate solution of the form \eqref{equiltr-pe}. Then, using Criterion \ref{pe-existence}, straightforward computations show that
$$
q_{12}=q_{13}=q_{23}=\frac{3y^2+3z^2-1}{2},\ \  \dot\alpha=\frac{c}{3mr^2},
$$
the equations in \eqref{sys-crit-pe} involving $\dot\alpha$ and $\ddot\alpha$ are identically satisfied,
and that the variables $y$ and $z$ must satisfy the equations
\begin{equation}
\label{ddotxy}
\begin{cases}
\dot y = u\cr
\dot z = v\cr
\dot{u}=F(y,z,u,v)y\cr
\dot{v}=F(y,z,u,v)z,
\end{cases}
\end{equation}
where
$$
F(y,z,u,v)=\frac{8m}{\sqrt{3}(1-y^2-z^2)^{\frac{1}{2}}(1+3y^2+3z^2)^{\frac{3}{2}}}-\frac{c^2}{9m^2(1-y^2-z^2)}
$$
$$
-\frac{u^2+v^2-(yv-uz)^2}{1-y^2-z^2}.
$$
From \eqref{ddotxy}, we can conclude that $\ddot y z=y\ddot z$, which implies that 
$$
y\dot z-z\dot y=k\ {\rm (constant)}.
$$
But, since from \eqref{angmom-yz} we have that $3m(y\dot z-z\dot y)=c_{yz}$, it follows that $k=c_{yz}/3m$. As, by Proposition \ref{integrals-pe}, $c_{yz}=0$, we must take $k=0$, so $y\dot z-z\dot y=0$, and therefore $\frac{d}{dt}\frac{y}{z}=0$ if $z$ does not take zero values, so $y(t)=\gamma z(t)$, where $\gamma$ is a constant. Moreover, since 
$$
\sin\alpha+\sin(\alpha+2\pi/3)+\sin(\alpha+4\pi/3)=\cos\alpha+\cos(\alpha+2\pi/3)+\cos(\alpha+4\pi/3)=0,
$$
it follows from \eqref{angmom-wy}, \eqref{angmom-wz}, \eqref{angmom-xy}, and \eqref{angmom-xz} that $c_{wy}=c_{wz}=c_{xy}=c_{xz}=0$, so the triangle has no rotation relative to the planes $wy, wz, xy, xz$, and $yz$.

If we denote $\delta=\gamma^2+1\ge 1$ (which implies that
$u=\gamma v$), substitute $y$ for $\gamma z$ (and therefore $u$ for $\gamma v$), make the change of variable $\bar z=\sqrt{\delta} z, \bar v=\sqrt{\delta} v$, and redenote $\bar z, \bar v$ by $z$ and $v$, respectively, system \eqref{ddotxy} reduces to the family of first order systems
\begin{equation}
\label{sys-L-pe}
\begin{cases}
\dot z=v\cr
\dot v=\Big[\frac{8m}{\sqrt{3}(1-z^2)^{1/2}(1+3 z^2)^{3/2}}-\frac{c^2}{9m^2(1-z^2)}-\frac{v^2}{1-z^2}\Big]z.
\end{cases}
\end{equation}
The fixed points of this system correspond to relative equilibria. An obvious one is $(z,v)=(0,0)$, which is a Lagrangian relative equilibrium rotating on a great circle of a great sphere of $\mathbb S^3$. The other fixed points are given by the polynomial equation of degree six
\begin{equation}\label{pol-Lagr-pe}
1728m^3(1-z^2)=c^2(1+3z^2)^3.
\end{equation}
This means that, for $m$ and $c$ fixed, there is a finite number of relative equilibria. Standard results of the theory of ordinary differential equations can now be applied to system \eqref{sys-L-pe} to prove the existence and uniqueness of analytic positive elliptic Lagrangian rotopulsators, for admissible initial conditions. This remark completes the proof.
\end{proof}

Although we proved the existence of the positive elliptic Lagrangian rotopulsators, it would be still interesting to learn more about their nature in terms of the energy constant. For this, we will first find the fixed points for the vector field of system \eqref{sys-L-pe} different from the obvious one, $(z,v)=(0,0)$, which is common to all equations in the family.
Using the energy relation \eqref{energy-pe}, obtaining the solutions of equation \eqref{pol-Lagr-pe} reduces to finding the zeroes of the family of polynomials
$$
P(z)=27(9m^4+h^2)z^8-18(15m^4+h^2)z^4-8h^2 z^2+75m^4-h^2.
$$
According to Descartes's rule of signs, we have to distinguish between two cases:

(i) $|h|<5\sqrt{3}m^2$, when, for every fixed values of the parameters, $P$ has either two positive roots or no positive root at all;

(ii) $|h|\ge 5\sqrt{3}m^2$, when, for every fixed values of the parameters, $P$ has exactly one positive root.

Case (ii) always leads to one fixed point since the unique positive root, $z_0:=z_0(m,h,\epsilon)$, has the property $|z_0|<1$, since $P(1)=48m^4>0$ and $P(0)=75m^4-h^2\le 0$. Then the corresponding eigenvalues $\lambda_{1,2}$ are given by the equation 
$$
\lambda+\frac{2h}{3m}-W(z_0, m, \epsilon)=0,
$$
where $W(z_0, m, \epsilon)$ is a finite number for every fixed value of the parameters. 
Independently of the values of $W(z_0, m, \epsilon)$, the eigenvalues show that $z$ varies for every orbit that is not a fixed point. Similar conclusions can be drawn in case (i).

Numerical experiments suggest that all the other orbits of system \eqref{sys-L-pe} are periodic, except for two 
homoclinic orbits. Since the periods of $\alpha$ and $z$ don't usually match, the periodic orbits generate
quasiperiodic positive elliptic Lagrangian rotopulsators, except for a negligible class, given by periodic positive elliptic Lagrangian rotopulsators (see Fig.\ \ref{Lagr-pe}).

\begin{figure}[htbp] 
   \centering
   \includegraphics[width=2in]{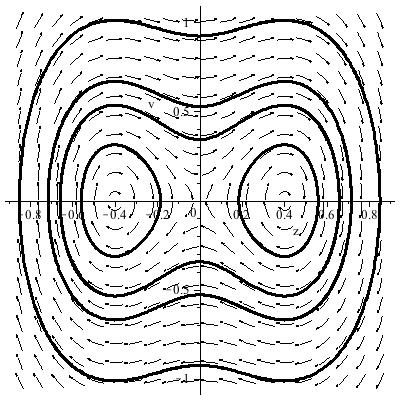}
   \caption{A typical phase portrait for system \eqref{sys-L-pe}.}
   \label{Lagr-pe}
\end{figure}

\begin{remark}
When $y$ or $z$ is a nonzero constant, the motion takes place on a 2-dimensional nongreat sphere. When $y\equiv0$ or $z\equiv0$, the motion is confined to a 2-dimensional great sphere, i.e.\ to $\mathbb S^2$. The latter case corresponds to a class of Lagrangian orbits for which we gave a complete classification in \cite{Diacu5}.
\end{remark}

\section{Positive elliptic-elliptic rotopulsators}

In this section we analyze the solutions given in Definition \ref{def-positive-elliptic-elliptic}. We first prove a criterion for finding such solutions, then obtain the conservation laws, and finally analyze two particular class of examples. In the first case we prove the existence of positive elliptic-elliptic rotopulsators in the 2-body problem in $\mathbb S^3$. In the second case we show that the positive elliptic-elliptic Lagrangian orbits of the 3-body problem in $\mathbb S^3$ are always relative equilibria, so they cannot form rotopulsators. 

\subsection{Criterion for positive elliptic-elliptic rotopulsators or relative equilibria}

The following result provides necessary and sufficient conditions for the existence of positive elliptic-elliptic rotopulsators or relative equilibria in $\mathbb S^3$ .

\begin{criterion}
\label{pee-existence}
A solution candidate of the form \eqref{positive-elliptic-elliptic} is a positive elliptic-elliptic rotopulsator for system \eqref{positive} if and only if
\begin{equation}
\label{alpha-and-beta}
\dot\alpha=\frac{c_1}{\sum_{i=1}^Nm_ir_i^2}, \ \ \ \dot\beta=\frac{c_2}{M-\sum_{i=1}^Nm_ir_i^2},
\end{equation}
with $c_1,c_2, M=\sum_{i=1}^Nm_i$ nonzero constants, there are at least two distinct indices $i,j\in\{1,2,\dots,N\}$ such that $q_{ij}$ is not constant, and the variables $r_1, r_2,\dots, r_N$ satisfy the first-order system of\ \! $4N$ equations,
\begin{equation}
\label{r-pee-new-new}
\begin{cases}
\dot r_i=s_i\cr
\dot{s}_i=G_i({\bf r}, s_i)\cr
r_i\ddot\alpha+2\dot{r}_i\dot\alpha=\sum_{\stackrel{j=1}{j\ne i}}^N
\frac{m_jr_j\sin(a_j-a_i)}{(1-q_{ij}^2)^{3/2}}\cr
\rho_i\ddot\beta+2\dot{\rho}_i\dot\beta=\sum_{\stackrel{j=1}{j\ne i}}^N
\frac{m_j\rho_j\sin(b_j-b_i)}{(1-q_{ij}^2)^{3/2}},
\end{cases}
\end{equation}
where $r_i^2+\rho_i^2=1, \ i=1,2,\dots, N$, ${\bf r}=(r_1, r_2,\dots, r_N)$, 
\begin{equation}
\begin{split}
G_i({\bf r}, s_i)=r_i(1-r_i^2)\bigg[\frac{c_1^2}{(\sum_{i=1}^Nm_ir_i^2)^2}-\frac{c_2^2}{(M-\sum_{i=1}^Nm_ir_i^2)^2}\bigg]-\frac{r_is_i^2}{1-r_i^2}\hspace{0.7cm}\\
+\sum_{\stackrel{j=1}{j\ne i}}^N\frac{m_j[r_j(1-r_i^2)\cos(a_i-a_j)-r_i(1-r_i^2)^{\frac{1}{2}}(1-r_j^2)^{\frac{1}{2}}\cos(b_i-b_j)]}{(1-q_{ij}^2)^{\frac{3}{2}}},
\end{split}
\end{equation}
$i=1,2,\dots, N$, and for any $i,j\in\{1,2,\dots,N\}$ with $i\ne j$,
\begin{equation}
\label{epsilonij}
q_{ij}=r_ir_j\cos(a_i-a_j)+(1-r_i^2)^{\frac{1}{2}}(1-r_j^2)^{\frac{1}{2}}\cos(b_i-b_j).
\end{equation}
If $q_{ij}$ are constant for all distinct $i,j\in\{1,2,\dots,N\}$, then the solution is a relative equilibrium. If $q_{ij}=\pm 1$ for some distinct $i,j\in\{1,2,\dots,N\}$, then such solutions don't exist.
\end{criterion}
\begin{proof}
Consider a candidate solution of the form \eqref{positive-elliptic-elliptic} for system \eqref{positive}. By expressing each $\rho_i$ in terms of $r_i$, $i=1,2,\dots,N$, we obtain that
$$
\rho_i=(1-r_i^2)^{\frac{1}{2}},\ \ \ \dot\rho_i=-\frac{r_i\dot r_i}{(1-r_i^2)^{\frac{1}{2}}},\ \ \
\ddot\rho_i=-\frac{\dot{r}_i^2+r_i(1-r_i^2)\ddot{r}_i}{(1-r_i^2)^{\frac{3}{2}}}.
$$
Then $q_{ij}$ takes the form \eqref{epsilonij}, and
$$
\dot q_{ij}=\dot{r}_i^2+r_i^2\dot\alpha^2+\frac{r_i^2\dot{r}_i^2}{1-r_i^2}+(1-r_i^2)\dot\beta^2.
$$
Substituting a solution candidate of the form \eqref{positive-elliptic-elliptic} into system \eqref{positive}, and using the above formulas, we obtain for the equations corresponding to $\ddot{w}_i$ and  $\ddot{x}_i$ the equations 
\begin{equation}
\label{r-pee}
\begin{split}
\ddot{r}_i=r_i(1-r_i^2)(\dot\alpha^2-\dot\beta^2)-\frac{r_i\dot{r}_i^2}{1-r_i^2}\hspace{3.5cm}\\
+\sum_{\stackrel{j=1}{j\ne i}}^N\frac{m_j[r_j(1-r_i^2)\cos(a_i-a_j)-r_i(1-r_i^2)^{\frac{1}{2}}(1-r_j^2)^{\frac{1}{2}}\cos(b_i-b_j)]}{(1-\epsilon_{ij}^2)^{\frac{3}{2}}},
\end{split}
\end{equation}
\begin{equation}
\label{alpha-pee}
r_i\ddot\alpha+2\dot{r}_i\dot\alpha=-\sum_{\stackrel{j=1}{j\ne i}}^N
\frac{m_jr_j\sin(a_i-a_j)}{(1-q_{ij}^2)^{\frac{3}{2}}}, \ \ i=1,2,\dots, N,
\end{equation} 
respectively, whereas for the equations corresponding to $\ddot{y}_i, \ddot{z}_i$, we find equations \eqref{r-pee} again as well as the equations
\begin{equation}
\label{beta-pee}
\rho_i\ddot\beta+2\dot{\rho}_i\dot\beta=-\sum_{\stackrel{j=1}{j\ne i}}^N
\frac{m_j\rho_j\sin(b_i-b_j)}{(1-q_{ij}^2)^{\frac{3}{2}}}, \ \ i=1,2,\dots, N.
\end{equation} 
We can solve equations \eqref{alpha-pee} the same way we solved equations \eqref{alpha-pe} and obtain
$$
\dot\alpha=\frac{c_1}{\sum_{i=1}^Nm_ir_i^2},
$$
where $c_1$ is an integration constant. To solve equations \eqref{beta-pee}, we proceed similarly, with the only change that for each $i=1,2,\dots,N$, the corresponding equation gets multiplied by $m_i(1-r_i^2)^{\frac{1}{2}}$ instead of $m_ir_i$, to obtain
after addition that
$$
\dot\beta=\frac{c_2}{M-\sum_{i=1}^Nm_ir_i^2},
$$
where $M=\sum_{i=1}^Nm_i$ and $c_2$ is an integration constant. Then equations \eqref{r-pee}, \eqref{alpha-pee}, \eqref{beta-pee} form system \eqref{r-pee-new-new}. Using the above expressions of $\dot\alpha$ and $\dot\beta$, we can conclude the same way as we did in the proof of Criterion 1 that \eqref{r-pee-new-new} is a first-order system of $4N$ equations with no constraints. The part of the criterion related to relative equilibria follows directly from Definition \ref{def-positive-elliptic-elliptic}. The nonexistence of such solutions if some $q_{ij}=\pm 1$ follows from the fact that at least a denominator cancels in the equations of motion. This remark completes the proof.
\end{proof}

\begin{remark}
It follows from \eqref{alpha-and-beta} that $\dot\alpha$ and $\dot\beta$ are connected by the relationship
\begin{equation}
\label{rel-alpha-beta}
\frac{c_1}{\dot\alpha}+\frac{c_2}{\dot\beta}=M,
\end{equation}
written under the assumption that $\alpha$ and $\beta$ are not constant. In particular, if $\alpha$ and $\beta$ differ only by an additive constant, they are linear functions of time, i.e.
\begin{equation*}
\dot\alpha=\dot\beta=\frac{c_1+c_2}{M}.
\end{equation*}
\end{remark}

\subsection{Conservation laws for positive elliptic-elliptic rotopulsators}

In addition to Criterion \ref{pee-existence}, we would also like to obtain the conservation laws specific to positive elliptic-elliptic rotopulsators. These laws follow by straightforward computations using the above proof, the integral of energy, and the six integrals of the total angular momentum.

\begin{proposition}\label{integrals-pee}
If system \eqref{positive} has a solution of the form \eqref{positive-elliptic-elliptic}, then the following expressions are constant:

--- energy:
\begin{equation}
\label{energy-pee}
h=\sum_{i=1}^N\frac{m_i\dot{r}_i^2}{2(1-r_i^2)}+\frac{c_1^2}{2\sum_{j=1}^Nm_jr_j^2}+\frac{c_2^2}{2(M-\sum_{j=1}^Nm_jr_j^2)}
-\sum_{1\le i<j\le N}\frac{m_im_jq_{ij}}{(1-q_{ij}^2)^{\frac{1}{2}}};
\end{equation}

--- total angular momentum relative to the $wx$-plane,
\begin{equation}
\label{cwx-pee}
c_{wx}=c_1,
\end{equation}
where $c_1\ne 0$ is the constant in the expression \eqref{alpha-and-beta} of $\dot\alpha$;

--- total angular momentum relative to the $wy$-plane, 
\begin{equation}
\label{cwy-pee}
\begin{split}
c_{wy}=\frac{1}{2}\sum_{i=1}^Nm_i\bigg[r_i(1-r_i^2)^{\frac{1}{2}}(\dot\alpha+\dot\beta)\sin(\alpha-\beta+a_i-b_i)\hspace{2cm}\\
+r_i(1-r_i^2)^{\frac{1}{2}}(\dot\alpha-\dot\beta)\sin(\alpha+\beta+a_i+b_i)-
\frac{\dot{r}_i}{(1-r_i^2)^{\frac{1}{2}}}\cos(\alpha-\beta+a_i-b_i)\\
-\frac{\dot{r}_i}{(1-r_i^2)^{\frac{1}{2}}}\cos(\alpha+\beta+a_i+b_i)\bigg];\hspace{3.3cm}
\end{split}
\end{equation}

--- total angular momentum relative to the $wz$-plane,
\begin{equation}
\label{cwz-pee}
\begin{split}
c_{wz}=\frac{1}{2}\sum_{i=1}^Nm_i\bigg[r_i(1-r_i^2)^{\frac{1}{2}}(\dot\alpha+\dot\beta)\cos(\alpha-\beta+a_i-b_i)\hspace{2cm}\\
-r_i(1-r_i^2)^{\frac{1}{2}}(\dot\alpha-\dot\beta)\cos(\alpha+\beta+a_i+b_i)
+\frac{\dot{r}_i}{(1-r_i^2)^{\frac{1}{2}}}\sin(\alpha-\beta+a_i-b_i)\\
-\frac{\dot{r}_i}{(1-r_i^2)^{\frac{1}{2}}}\sin(\alpha+\beta+a_i+b_i)\bigg];\hspace{3.3cm}
\end{split}
\end{equation}

--- total angular momentum relative to the $xy$-plane,
\begin{equation}
\label{cxy-pee}
\begin{split}
c_{xy}=-\frac{1}{2}\sum_{i=1}^Nm_i\bigg[r_i(1-r_i^2)^{\frac{1}{2}}(\dot\alpha+\dot\beta)\cos(\alpha-\beta+a_i-b_i)\hspace{1.8cm}\\
+r_i(1-r_i^2)^{\frac{1}{2}}(\dot\alpha-\dot\beta)\cos(\alpha+\beta+a_i+b_i)+
\frac{\dot{r}_i}{(1-r_i^2)^{\frac{1}{2}}}\sin(\alpha-\beta+a_i-b_i)\\
+\frac{\dot{r}_i}{(1-r_i^2)^{\frac{1}{2}}}\sin(\alpha+\beta+a_i+b_i)\bigg];\hspace{3.3cm}
\end{split}
\end{equation}

--- total angular momentum relative to the $xz$-plane,
\begin{equation}
\label{cxz-pee}
\begin{split}
c_{xz}=\frac{1}{2}\sum_{i=1}^Nm_i\bigg[r_i(1-r_i^2)^{\frac{1}{2}}(\dot\alpha+\dot\beta)\sin(\alpha-\beta+a_i-b_i)\hspace{2cm}\\
-r_i(1-r_i^2)^{\frac{1}{2}}(\dot\alpha-\dot\beta)\sin(\alpha+\beta+a_i+b_i)-
\frac{\dot{r}_i}{(1-r_i^2)^{\frac{1}{2}}}\cos(\alpha-\beta+a_i-b_i)\\
+\frac{\dot{r}_i}{(1-r_i^2)^{\frac{1}{2}}}\cos(\alpha+\beta+a_i+b_i)\bigg];\hspace{3.3cm}
\end{split}
\end{equation}

--- total angular momentum relative to the $yz$-plane,
\begin{equation}
\label{cwx-pee}
c_{yz}=c_2,
\end{equation}
where $c_2\ne 0$ is the constant in the expression \eqref{alpha-and-beta} of $\dot\beta$.
\end{proposition}

The above result could also be used to prove the nonexistence of some candidates for positive elliptic rotopulsating orbits, by showing that at least one of the above conservation laws is violated.

\subsection{Positive elliptic-elliptic rotopulsators for $N=2$}

We further show the existence of a class of positive elliptic-elliptic rotopulsators of the 2-body problem in $\mathbb S^3$. These binary systems rotate relative to the planes $wx$ and $yz$, but have no rotations with respect to the planes $wy, wz, xy$, and $xz$.

Consider two equal masses, $m_1=m_2=:m>0$, and a candidate solution of the form
\begin{equation}\label{2-pee}
\begin{split}
{\bf q}=({\bf q}_1,{\bf q}_2),\ {\bf q}_i=(w_i,x_i,y_i,z_i), \ i=1,2,\hspace{3cm}\\
w_1=r(t)\cos\alpha(t), \ x_1=r(t)\sin\alpha(t),\ y_1=\rho(t)\cos\beta(t),\ z_1=\rho(t)\sin\beta(t),\\
w_2=r(t)\cos[\alpha(t)+\pi], \ x_2=r(t)\sin[\alpha(t)+\pi],\hspace{2.4cm}\\ y_2=\rho(t)\cos\beta(t),\ z_2=\rho(t)\sin\beta(t),\hspace{3.3cm}
\end{split}
\end{equation}
with $\alpha$ and $\beta$ nonconstant functions and $r^2+\rho^2=1$. We can prove now the following result, which shows the existence of rotopulsators of the above form.

\begin{proposition}
Consider the curved $2$-body problem in $\mathbb S^3$ given by system \eqref{positive} with $N=2$. Then, except for a negligible class of relative equilibria, every candidate solution of the form \eqref{2-pee} is a positive elliptic-elliptic rotopulsator, which rotates relative to the planes $wx$ and $yz$, but has no rotation with respect to the planes $wy, wz, xy,$ and $xz$.
\end{proposition}
\begin{proof}
From Criterion \ref{pee-existence} we notice that
$$
q_{12}=2r^2-1,\ \dot\alpha=\frac{c_1}{2mr^2},\ 
\dot\beta=\frac{c_2}{2m(1-r^2)},
$$
and that a candidate solution of the form \eqref{2-pee}
leads to the family of first-order systems
\begin{equation}\label{2b-pee}
\begin{cases}
\dot r = s\cr
\dot s = r(1-r^2)\Big[\frac{c_1}{4m^2r^4}-\frac{c_2}{4m^2(1-r^2)^2}\Big]-\frac{rs^2}{1-r^2}-\frac{m}{4r^2(1-r^2)^{1/2}},
\end{cases}
\end{equation} 
since the equations involving $\dot\alpha, \ddot\alpha$ and $\dot\beta,\ddot\beta$, respectively, in system \eqref{r-pee-new-new} are identically satisfied. Standard results of the theory of differential equations prove the existence and uniqueness of analytic solutions for nonsingular initial conditions attached to the above system. For fixed values of $m>0, c_1\ne 0,$ and $c_2\ne 0$, the number of fixed points of system \eqref{2b-pee} is obviously finite, so the set of fixed points is negligible when these constants vary. Therefore, except for the negligible set of orbits corresponding to fixed points, which are relative equilibria because $r$ is constant, all the solutions of the form \eqref{2-pee} are rotopulsators since $r$ (and consequently the mutual distance $q_{12}$ between the bodies) varies.

It follows from Proposition \ref{integrals-pee} that 
$c_{wx}=c_1, c_{yz}=c_2$, with $c_1,c_2\ne 0$, and that $c_{wy}=c_{wz}=c_{xy}=c_{xz}=0$, which proves that the binary system rotates relative to the planes $wx$ and $yz$, but has no rotation with respect to the planes $wy, wz, xy,$ and $xz$. This remark completes the proof.
\end{proof}

\subsection{Positive elliptic-elliptic Lagrangian relative equilibria}\label{pee-Lagr}

We further prove that the positive elliptic-elliptic Lagrangian orbits of the curved 3-body problem in $\mathbb S^3$ (i.e.\ equilateral triangles having two rotations, one with respect to the plane  $wx$ and the other relative to the plane $yz$) are necessarily relative equilibria, and cannot be rotopulsators. Moreover, these orbits have no rotations relative to the other base planes.

Consider three equal masses, $m_1=m_2=m_3=:m>0$, and a candidate solution of the form 
\begin{equation}\label{equiltr-pee}
\begin{gathered}
{\bf q}=({\bf q}_1,{\bf q}_2,{\bf q}_3), \ \ {\bf q}_i=(w_i,x_i,y_i,z_i), \ \ i=1,2,3,\\
w_1=r(t)\cos\alpha(t),\ x_1=r(t)\sin\alpha(t),\
y_1=\rho(t)\cos\beta(t),\ z_1=\rho(t)\sin\beta(t),\\
w_2=r(t)\cos[\alpha(t)+2\pi/3],\ x_2=r(t)\sin[\alpha(t)+2\pi/3],\\
y_2=\rho(t)\cos[\beta(t)+2\pi/3],\ z_2=\rho(t)\sin[\beta(t)+2\pi/3],\\
w_3=r(t)\cos[\alpha(t)+4\pi/3],\ x_3=r(t)\sin[\alpha(t)+4\pi/3],\\
y_3=\rho(t)\cos[\beta(t)+4\pi/3],\ z_3=\rho(t)\sin[\beta(t)+4\pi/3],
\end{gathered}
\end{equation}
with $\alpha$ and  $\beta$ nonconstant functions and $r^2+\rho^2=1$. We can now prove the following result.

\begin{proposition}
\label{prop-pee-L}
Consider the curved $3$-body problem in $\mathbb S^3$ given by system \eqref{positive} with $N=3$. Then every candidate solution of the form \eqref{equiltr-pee} is a positive elliptic-elliptic Lagrangian relative equilibrium,
which rotates relative to the planes $wx$ and $yz$, but has no rotation with respect to the planes $wy, wz, xy,$ and $xz$. 
\end{proposition}
\begin{proof} To prove this result, we first apply Criterion \ref{pee-existence} to a candidate solution of the form \eqref{equiltr-pee} for system \eqref{positive}. Straightforward computations show that 
$$
q_{12}=q_{13}=q_{23}=-1/2,
$$ 
which means that the sides of the equilateral triangle don't vary in time, so if this solution candidate proves to exist, then it is necessarily a positive elliptic-elliptic Lagrangian relative equilibrium. Notice further that
\begin{equation}
\label{1&2eq}
\dot\alpha=\frac{c_1}{3mr^2},\ \ \dot\beta=\frac{c_2}{3m(1-r^2)},
\end{equation}
the equations in system \eqref{r-pee-new-new} involving $\dot\alpha, \ddot\alpha$ and $\dot\beta, \ddot\beta$, respectively, are identically satisfied, and that system \eqref{r-pee-new-new} thus reduces to the family of first-order systems
\begin{equation}
\label{Lagrangian-ee}
\begin{cases}
\dot{r}=u,\cr
\dot{u}=\frac{c_1^2(1-r^2)}{9m^2r^3}-
\frac{r(9m^2u^2+c_2^2)}{9m^2(1-r^2)}.
\end{cases}
\end{equation} 

As in the proof of Proposition \ref{prop-pe-L}, the existence and uniqueness of analytic positive elliptic-elliptic rotopulsators for admissible initial conditions follows. From Proposition \ref{integrals-pee} we can conclude that $c_{wx}=c_1, c_{yz}=c_2$, with $c_1,c_2\ne 0$, and that $c_{wy}=c_{wz}=c_{xy}=c_{xz}=0$, which proves that the positive elliptic-elliptic Lagrangian relative equilibria rotate relative to the planes $wx$ and $yz$, but have no rotation with respect to the planes $wy, wz, xy,$ and $xz$. This remark completes the proof.
\end{proof}

\begin{remark}
As noted at the end of Section 4, since $r$ varies, the above relative equilibria cannot be generated from an
element of the underlying torus $SO(2)\times SO(2)$
of the Lie group $SO(4)$. Suitable rotations of the coordinate system, however, would make this possible, in which case $r$ would become constant for each specific solution. 
\end{remark}

To get some insight into the nature of these relative equilibria, let us first find the fixed points of system \eqref{Lagrangian-ee},
which occur for
$$
(c_1^2-c_2^2)r^4-2c_1^2r^2+c_1^2=0.
$$
For fixed values of $c_1$ and $c_2$, with $c_1\ne c_2$, there are at most two fixed points, $r_1=\big(\frac{c_1}{c_1-c_2}\big)^{1/2}$ and $r_2=\big(\frac{c_1}{c_1+c_2}\big)^{1/2}$, only one of which is between $0$ and $1$, whereas for $c_1=c_2$ there is a single fixed point, namely $r_0=1/\sqrt{2}$. So in both cases the fixed point is unique.

Numerical experiments suggest that the phase space picture of system \eqref{Lagrangian-ee} looks like in Fig.\eqref{Lagr-elliptic-elliptic}, with $r$ periodic, a fact in agreement with the Lie theory applied to the group $SO(4)$. Consequently, the positive elliptic-elliptic Lagrangian orbits are expected to be quasiperiodic, since the periods of $\alpha$ and $\beta$ differ, in general, except for a negligible set corresponding to periodic orbits.

\begin{figure}[htbp] 
   \centering
   \includegraphics[width=2in]{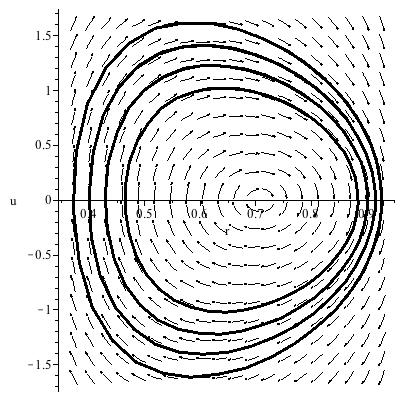}
   \caption{A typical flow of system \eqref{Lagrangian-ee} for $c_1=c_2=1$. For $c_1\ne c_2$, the flow looks qualitatively similar.}
   \label{Lagr-elliptic-elliptic}
\end{figure}

\begin{remark}
From relations \eqref{1&2eq} we can see how the angular velocities $\dot\alpha$ and $\dot\beta$ vary relative to $r$: when $r$ is close to 0, $|\dot\alpha|$ is large, while $|\dot\beta|$ is small, and the other way around when $r$ is close to 1. Moreover, the expressions of $\dot\alpha$ and $\dot\beta$ are consistent with relation \eqref{rel-alpha-beta}.
\end{remark}

\begin{remark}
If we initially assume the masses to be distinct, but the bodies to correspond to the same function $r$, it follows that the solutions don't exist, which means that the masses must be always equal. If we additionally take distinct functions $r_1, r_2, r_3$, then
$$
q_{ij}=-\frac{r_ir_j+(1-r_i^2)^{1/2}(1-r_j^2)^{1/2}}{2}=\frac{(r_i-r_j)^2+[(1-r_i^2)^{1/2}-(1-r_j^2)^{1/2}]^2-2}{4},
$$
which implies that the triangle is not necessarily equilateral, 
so the orbit may not be Lagrangian at all.
\end{remark}

\begin{remark}
The energy relation \eqref{energy-pee} takes the form
$$
h=\frac{3m\dot{r}^2}{2(1-r^2)}+\frac{1}{6m^2}\bigg[\frac{c_1^2}{r^2}+\frac{c_2^2}{1-r^2}\bigg]+\sqrt{3}m^2,
$$
so the energy constant, $h$, is always positive.
\end{remark}

\begin{remark}
When $\dot\alpha=\dot\beta$, we have $r^2=\frac{c_1}{c_1+c_2}$, which means that $r$ must be constant, if it exists, so 
$$
\dot\alpha=\dot\beta=\frac{c_1+c_2}{3m}.
$$
In this case the integrals \eqref{energy-pee}, \eqref{cwy-pee}, \eqref{cwz-pee}, \eqref{cxy-pee}, and \eqref{cxz-pee} are constant, as expected. 
\end{remark}


\section{Qualitative behaviour of rotopulsators in $\mathbb S^3$}

In this section we will prove a result that describes the qualitative behaviour of rotopulsators in $\mathbb S^3$.
For this purpose, we first briefly introduce an object familiar to geometric topologists and present some of its properties.


\subsection{Clifford tori} The 2-dimensional manifold defined by
\begin{equation}
\label{Cliff}
{\bf T}_{r\rho}:=\{(w,x,y,z)\in\mathbb R^4\ |\ r^2+\rho^2=1, \ \ 0\le \theta,\phi < 2\pi\},
\end{equation}
where $w=r\cos\theta, x=r\sin\theta, y=\rho\cos\phi$, and
$z=\rho\sin\phi$, with $r,\rho\ge 0$, is called a Clifford torus,
and it has zero Gaussian curvature.
Since the distance from $0$ to every point of ${\bf T}_{r\rho}$ is 1, it follows that Clifford tori are contained in $\mathbb S^3$.
When $r$ (and, consequently, $\rho$) takes all the values between 0 and 1, the family of Clifford tori such defined foliates $\mathbb S^3$. Each Clifford torus splits $\mathbb S^3$ into two solid tori and forms the boundary between them. The two solid tori are congruent only when $r=\rho=1/\sqrt{2}$.

We have previously shown that relative equilibria in $\mathbb S^3$ rotate on Clifford tori, \cite{Diacu3}, \cite{Diacu4}. We will next prove that, at every moment in time, a rotopulsator passes through a different Clifford torus of any given foliation of $\mathbb S^3$. In other words, rotopulsators cannot be generated by an element of any underlying subgroup $SO(2)\times SO(2)$ of the Lie group $SO(4)$.

\subsection{Geometry and dynamics of rotopulsators in $\mathbb S^3$}

We can now state and prove the following result, which describes the motion of the bodies relative to foliations of $\mathbb S^3$ with Clifford tori.

\begin{theorem}
Consider a positive elliptic or a positive elliptic-elliptic rotopulsator of the curved $N$-body problem in $\mathbb S^3$. Then, for any foliation $({\bf T}_{r\rho})_{0\le r,\rho\le 1}$ of $\mathbb S^3$ given by Clifford tori, it is impossible that the trajectory of each body is contained for all time in some Clifford torus. In other words, for any such foliation, there is at least one body whose trajectory intersects a continuum of Clifford tori.
\end{theorem}
\begin{proof}
Let us assume that there exists a foliation $({\bf T}_{r\rho})_{0\le r,\rho\le 1}$ of $\mathbb S^3$ with Clifford tori for which a solution of the form \eqref{positive-elliptic} or \eqref{positive-elliptic-elliptic} behaves such that the trajectory of each body is confined to a Clifford torus. We will prove that under this hypothesis such a solution must be a relative equilibrium.

Let us first prove this property for positve elliptic-elliptic rotopulsators, i.e.\ solutions of the form \eqref{positive-elliptic-elliptic}. If the body $m_i$, whose solution is described by 
$$
w_i=r_i\cos(\alpha+a_i), x_i=r_i\sin(\alpha+a_i), y_i=\rho_i\cos(\beta+b_i), z_i=\rho_i\sin(\beta+b_i),
$$
with $r_i^2+\rho_i^2=1$, is confined to the Clifford torus ${\bf T}_{r^\sharp\rho^\sharp}$, with $r^\sharp,\rho^\sharp$ constant, then
$r_i=r^\sharp$ and $\rho_i=\rho^\sharp$ are also constant. Similarly,  if the body $m_j$, whose solution is described by 
$$
w_j=r_i\cos(\alpha+a_j), x_j=r_i\sin(\alpha+a_j), y_i=\rho_j\cos(\beta+b_j), z_j=\rho_j\sin(\beta+b_j),
$$
with $r_j^2+\rho_j^2=1$, is confined to the Clifford torus ${\bf T}_{r^\flat\rho^\flat}$, with $r^\flat,\rho^\flat$ constant, then $r_j=r^\flat$ and $\rho_j=\rho^\flat$ are also constant. As a result, 
$$
q_{ij}=r^\sharp r^\flat\cos(a_i-a_j)+\rho^\sharp\rho^\flat\cos(b_i-b_j),
$$
which is constant. So all the mutual distances are constant, therefore the solution is a relative equilibrium. 

For positive elliptic rotopulsators, we can use Remark \ref{def-nonstandard}, and notice that they are positive elliptic-elliptic rotopulsators with $\beta\equiv 0$. But 
since $\beta$ does not occur anyway in the above expression of $q_{ij}$, it won't show up for $\beta\equiv 0$ either, so the mutual distances of such orbits are also constant.

Since the foliation of $\mathbb S^3$ with Clifford tori forms a continuum of surfaces, the last part of the theorem follows. This remark completes the proof.
\end{proof}

\section{Negative elliptic rotopulsators}

In this section we analyze the solutions given in Definition \ref{def-negative-elliptic}. We first prove a criterion for finding such solutions, then obtain the conservation laws, and finally discuss a particular class of examples, namely the negative elliptic Lagrangian rotopulsators of the 3-body problem in $\mathbb H^3$. 

\subsection{Criterion for negative elliptic rotopulsators or relative equilibria}

The following result provides necessary and sufficient conditions for the existence of negative elliptic rotopulsators or relative equilibria in $\mathbb H^3$.

\begin{criterion}\label{ne-existence}
A solution candidate of the form \eqref{negative-elliptic} is a positive elliptic rotopulsator for system \eqref{negative} if and only if
\begin{equation}
\dot\alpha=\frac{b}{\sum_{j=1}^Nm_jr_i^2}, 
\end{equation}
where $b\ne 0$ is a constant, there are at least two distinct indices $i,j\in\{1,2,\dots,N\}$ such that $q_{ij}$ is not constant, and the variables $y_i, z_i, r_i, \ i=1,2,\dots, N,$ satisfy the first-order system of\ \! $5N$ equations (with\ \! $N$ constraints: $r_i^2+y_i^2=z_i^2-1$, $i=1,2,\dots,N$), 
\begin{equation}
\label{sys-crit-ne}
\begin{cases}
\dot y_i=u_i\cr
\dot z_i=v_i\cr
\dot{u}_i=\sum_{\stackrel{j=1}{j\ne i}}^N\frac{m_j(y_j+q_{ij}y_i)}{(q_{ij}^2-1)^{3/2}}+H_i({\bf y},{\bf z}, u_i, v_i)y_i\cr
\ddot{z}_i=\sum_{\stackrel{j=1}{j\ne i}}^N\frac{m_j(z_j+q_{ij}z_i)}{(q_{ij}^2-1)^{3/2}}+H_i({\bf y},{\bf z}, u_i, v_i)z_i,\cr
r_i\ddot\alpha+2\dot{r}_i\dot\alpha=\sum_{\stackrel{j=1}{j\ne i}}^N
\frac{m_jr_j\sin(a_j-a_i)}{(\mu_{ij}^2-1)^{3/2}},
\end{cases}
\end{equation}
where ${\bf y}=(y_1,y_2,\dots,y_N),\ {\bf z}=(z_1,z_2,\dots,z_N)$,
\begin{equation}
H_i({\bf y},{\bf z}, u_i, v_i):=\frac{[(y_iv_i-z_iu_i)^2+v_i^2-u_i^2]}{z_i^2-y_i^2-1}+\frac{b^2(z_i^2-y_i^2-1)}{[\sum_{j=1}^Nm_j(z_j^2-y_j^2-1)]^2},
\end{equation}
$i=1,2,\dots, N,$ and, for any $i,j\in\{1,2,\dots,N\}$, 
$$
q_{ij}=(z_i^2-y_i^2-1)^{\frac{1}{2}}(z_j^2-y_j^2-1)^{\frac{1}{2}}\cos(a_i-a_j)+y_iy_j-z_iz_j.
$$
If the quantities $q_{ij}$ are constant for all distinct indices $i,j\in\{1,2,\dots,N\}$, then the solution is a relative equilibrium. If $q_{ij}=\pm 1$ for some distinct $i,j\in\{1,2,\dots,N\}$, then such solutions don't exist.
\end{criterion}
\begin{proof}
Consider a solution candidate of the form \eqref{negative-elliptic} subject to the above initial conditions. Then, for any $i,j\in\{1,2,\dots,N\}$, we obtain the above expression for $q_{ij}$
and, for any $i=1,2,\dots, N$, we find that 
$$
\dot q_{ij}=\frac{(y_i\dot{z}_i-z_i\dot{y}_i)^2+\dot{z}_i^2-\dot{y}_i^2+(z_i^2-y_i^2-1)^2\dot\alpha^2}{z_i^2-y_i^2-1}.
$$
For all $\ i=1,2,\dots, N$, each $r_i$ can be expressed in terms of $y_i$ and $z_i$ to obtain
$$
r_i=(z_i^2-y_i^2-1)^{\frac{1}{2}},\ \ \
\dot{r}_i=\frac{z_i\dot{z}_i-y_i\dot{y}_i}{(z_i^2-y_i^2-1)^{\frac{1}{2}}},
$$
$$
\ddot{r}_i=\frac{(z_i^2-y_i^2-1)(z_i\ddot{z}_i-y_i\ddot{y}_i)+\dot{y}_i^2-\dot{z}_i^2-(y_i\dot{z}_i-z_i\dot{y}_i)^2}{(z_i^2-y_i^2-1)^{\frac{3}{2}}}.
$$
Substituting a solution of the form \eqref{negative-elliptic} into system \eqref{negative} and employing the above formulas, we obtain for the equations corresponding to $\ddot{y}_i$ and $\ddot{z}_i$ that
\begin{equation}
\label{neg-ell-yi}
\ddot{y}_i=\sum_{\stackrel{j=1}{j\ne i}}^N\frac{m_j(y_j+q_{ij}y_i)}{(q_{ij}^2-1)^{\frac{3}{2}}}+\frac{[(y_i\dot{z}_i-z_i\dot{y}_i)^2+\dot{z}_i^2-\dot{y}_i^2]y_i}{z_i^2-y_i^2-1}+(z_i^2-y_i^2-1)y_i\dot\alpha^2,
\end{equation}
\begin{equation}
\label{neg-ell-zi}
\ddot{z}_i=\sum_{\stackrel{j=1}{j\ne i}}^N\frac{m_j(z_j+q_{ij}z_i)}{(q_{ij}^2-1)^{\frac{3}{2}}}+\frac{[(y_i\dot{z}_i-z_i\dot{y}_i)^2+\dot{z}_i^2-\dot{y}_i^2]z_i}{z_i^2-y_i^2-1}+(z_i^2-y_i^2-1)z_i\dot\alpha^2,
\end{equation}
whereas for the equations corresponding to $\ddot{w}_i$ and $\ddot{x}_i$, after some long computations that also use \eqref{neg-ell-yi} and \eqref{neg-ell-zi}, we are led either to identities or to the equations
\begin{equation}
\label{alpha-ne}
r_i\ddot\alpha+2\dot{r}_i\dot\alpha=\sum_{\stackrel{j=1}{j\ne i}}^N
\frac{m_jr_j\sin(a_j-a_i)}{(q_{ij}^2-1)^{\frac{3}{2}}}, \ \ i=1,2,\dots, N.
\end{equation} 
Then the same as in the proof of Criterion 1 we obtain that
$$
\dot\alpha=\frac{b}{\sum_{i=1}^Nm_ir_i^2}=\frac{b}{\sum_{i=1}^Nm_i(z_i^2-y_i^2-1)},
$$
where $b$ is an integration constant. Then, again as in the proof of Criterion 1, we can prove that
equations \eqref{neg-ell-yi}, \eqref{neg-ell-zi}, \eqref{alpha-ne} lead to the first order system \eqref{sys-crit-ne}. The part of the criterion related to relative equilibria follows directly from Definition \ref{def-negative-elliptic}. The nonexistence of such solutions if some $q_{ij}=\pm 1$ follows from the fact that at least a denominator cancels in the equations of motion. This remark completes the proof.
\end{proof}

\subsection{Conservation laws for negative elliptic rotopulsating orbits.}

In addition to Criterion \ref{ne-existence}, we would also like to obtain the conservation laws specific to negative elliptic rotopulsating orbits. They follow by straightforward computations using the above proof, the integral of energy, and the six integrals of the total angular momentum.

\begin{proposition}\label{integrals-ne}
If system \eqref{negative} has a solution of the form \eqref{negative-elliptic}, then the following expressions are constant:

--- energy,
\begin{equation}
\label{energy-ne}
\begin{split}
h=\sum_{i=1}^N\frac{m_i[(y_i\dot{z}_i-z_i\dot{y}_i)^2+\dot{z}_i^2-\dot{y}_i^2]}{2(z_i^2-y_i^2-1)}\hspace{1cm}\\
+\frac{b^2}{2\sum_{i=1}^Nm_i(z_i^2-y_i^2-1)}
+\sum_{1\le i<j\le N}\frac{m_im_jq_{ij}}{(q_{ij}^2-1)^{\frac{1}{2}}};
\end{split}
\end{equation}

--- total angular momentum relative to the $wx$-plane,
\begin{equation}
c_{wx}=b;
\end{equation}

--- total angular momentum relative to the $wy$-plane,
\begin{equation}
\label{angmom-wy-ne}
\begin{split}
c_{wy}=\sum_{i=1}^Nm_i\bigg[(z_i^2-y_i^2-1)^{\frac{1}{2}}\dot{y}_i+\frac{(y_i\dot{y}_i-z_i\dot{z}_i)y_i}{(z_i^2-y_i^2-1)^{\frac{1}{2}}}\bigg]\cos(\alpha+a_i)\\
+\frac{b}{\sum_{i=1}^Nm_i(z_i^2-y_i^2-1)}\sum_{i=1}^Nm_i(z_i^2-y_i^2-1)^{\frac{1}{2}}y_i\sin(\alpha+a_i);
\end{split}
\end{equation}

--- total angular momentum relative to the $wz$-plane,
\begin{equation}
\label{angmom-wz-ne}
\begin{split}
c_{wz}=\sum_{i=1}^Nm_i\bigg[(z_i^2-y_i^2-1)^{\frac{1}{2}}\dot{z}_i+\frac{(y_i\dot{y}_i-z_i\dot{z}_i)z_i}{(z_i^2-y_i^2-1)^{\frac{1}{2}}}\bigg]\cos(\alpha+a_i)\\
+\frac{b}{\sum_{i=1}^Nm_i(z_i^2-y_i^2-1)}\sum_{i=1}^Nm_i(z_i^2-y_i^2-1)^{\frac{1}{2}}z_i\sin(\alpha+a_i);
\end{split}
\end{equation}

--- total angular momentum relative to the $xy$-plane,
\begin{equation}
\label{angmom-xy-ne}
\begin{split}
c_{xy}=\sum_{i=1}^Nm_i\bigg[(z_i^2-y_i^2-1)^{\frac{1}{2}}\dot{y}_i+\frac{(y_i\dot{y}_i-z_i\dot{z}_i)y_i}{(z_i^2-y_i^2-1)^{\frac{1}{2}}}\bigg]\sin(\alpha+a_i)\\
-\frac{b}{\sum_{i=1}^Nm_i(z_i^2-y_i^2-1)}\sum_{i=1}^Nm_i(z_i^2-y_i^2-1)^{\frac{1}{2}}y_i\cos(\alpha+a_i);
\end{split}
\end{equation}

--- total angular momentum relative to the $xz$-plane,
\begin{equation}
\label{angmom-xz-ne}
\begin{split}
c_{xz}=\sum_{i=1}^Nm_i\bigg[(z_i^2-y_i^2-1)^{\frac{1}{2}}\dot{z}_i+\frac{(y_i\dot{y}_i-z_i\dot{z}_i)z_i}{(z_i^2-y_i^2-1)^{\frac{1}{2}}}\bigg]\sin(\alpha+a_i)\\
-\frac{b}{\sum_{i=1}^Nm_i(z_i^2-y_i^2-1)}\sum_{i=1}^Nm_i(z_i^2-y_i^2-1)^{\frac{1}{2}}z_i\cos(\alpha+a_i);
\end{split}
\end{equation}

--- total angular momentum relative to the $yz$-plane,
\begin{equation}
\label{angmom-yz-ne}
c_{yz}=0.
\end{equation}
\end{proposition}

\subsection{Negative elliptic Lagrangian rotopulsators}

We further provide a class of specific examples of negative elliptic rotopulsators of the curved 3-body problem, namely Lagrangian orbits in $\mathbb H^3$. These systems rotate relative to the plane $wx$, but have no rotations relative to the other base planes.

Consider three equal masses, $m_1=m_2=m_3=:m$, and a candidate solution of the form
\begin{equation}\label{equiltr-ne}
\begin{split}
{\bf q}=({\bf q}_1,{\bf q}_2,{\bf q}_3), \ \ {\bf q}_i=(w_i,x_i,y_i,z_i), \ \ i=1,2,3, \hspace{1.2cm}\\
w_1=r(t)\cos\alpha(t),\ x_1=r(t)\sin\alpha(t),\
y_1=y(t),\ z_1=z(t),\\
w_2=r(t)\cos[\alpha(t)+2\pi/3],\ x_2=r(t)\sin[\alpha(t)+2\pi/3],\hspace{0.5cm}\\
y_2=y(t),\ z_2=z(t),\hspace{3.3cm}\\
w_3=r(t)\cos[\alpha(t)+4\pi/3],\ x_3=r(t)\sin[\alpha(t)+4\pi/3],\hspace{0.5cm}\\
y_3=y(t),\ z_3=z(t).\hspace{3.3cm}
\end{split}
\end{equation}

With the help of Criterion \ref{ne-existence}, we can now show that, in general, these are indeed solutions of system \eqref{negative}.

\begin{proposition}
Consider the curved $3$-body problem in $\mathbb H^3$ given by system \eqref{negative}. Then, except for a negligible set of solutions formed by negative elliptic Lagrangian relative equilibria, every candidate solution of the form \eqref{equiltr-ne} is a negative elliptic Lagrangian rotopulsator, which rotates relative to the plane $wx$, but has no rotation with respect to the planes $wy, wz, xy, xz$, and $yz$.
\end{proposition}

\begin{proof}
Let us consider a candidate solution of the form \eqref{equiltr-ne}. Then, using Criterion \ref{ne-existence}, straightforward computations show that
$$
q_{12}=q_{13}=q_{23}=\frac{3y^2-3z^2+1}{2},\ \  \dot\alpha=\frac{b}{3mr^2},
$$
the equations of system \eqref{sys-crit-ne} involving $\dot\alpha, \ddot\alpha$ are identically satisfied, and the variables $y$ and $z$ satisfy the system
\begin{equation}
\label{ddotxy-ne}
\begin{cases}
\dot y = u\cr
\dot z = v\cr
\dot{u} = G(y,z,u,v)y\cr
\dot{v} = G(y,z,u,v)z,
\end{cases}
\end{equation}
where
$$
G(y,z,u,v)=\frac{(yv-uz)^2+v^2-u^2}{z^2-y^2-1}
+\frac{b^2}{9m^2(z^2-y^2-1)}
$$
$$
-\frac{8m}{\sqrt{3}(z^2-y^2-1)^{\frac{1}{2}}(3z^2-3y^2+1)^{\frac{3}{2}}}.
$$
From \eqref{ddotxy-ne}, we can conclude that $\ddot y z=y\ddot z$, which implies that 
$$
y\dot z-z\dot y=k\ {\rm (constant)}.
$$
But, since from \eqref{angmom-yz-ne} we have that $3m(y\dot z-z\dot y)=c_{yz}$, it follows that $k=c_{yz}/3m$.

Notice that the energy relation \eqref{energy-ne} takes the form
$$
\frac{3m[(y\dot z-\dot yz)^2+\dot z^2-\dot y^2]}{2(z^2-y^2-1)}+\frac{b^2}{6m(z^2-y^2-1)}+\frac{\sqrt{3}m^2(3y^2-3z^2+1)}{(z^2-y^2-1)^{\frac{1}{2}}(3z^2-3y^2+1)^{\frac{1}{2}}}=h,
$$
which implies that $G$ can be written as
$$
G(y,z)=\frac{2h}{3m}-\frac{2m[5-9(y^2-z^2)^2]}{\sqrt{3}(z^2-y^2-1)^{\frac{1}{2}}(3z^2-3y^2+1)^{\frac{3}{2}}}.
$$
Since 
$$
\sin\alpha+\sin(\alpha+2\pi/3)+\sin(\alpha+4\pi/3)=\cos\alpha+\cos(\alpha+2\pi/3)+\cos(\alpha+4\pi/3)=0,
$$
it follows from \eqref{angmom-wy-ne}, \eqref{angmom-wz-ne}, \eqref{angmom-xy-ne}, and \eqref{angmom-xz-ne} that $c_{wy}=c_{wz}=c_{xy}=c_{xz}=0$, so the triangle has no rotation relative to the planes $wy, wz, xy$, and $xz$. Since there is no rotation relative to the plane $yz$ either, i.e.\ $c_{yz}=0$, we have $k=0$, so $y\dot z-z\dot y=0$, and if we assume that $z$ does not take zero values, we can conclude that $\frac{d}{dt}\frac{y}{z}=0$, so $y(t)=\gamma z(t)$, where $\gamma$ is a constant. Let us now denote $\epsilon=1-\gamma^2$. Notice that since a point $(w,x,y,z)$ on $\mathbb H^3$ satisfies the equation $w^2+x^2+y^2-z^2=-1$ and $z\ge 1$, we necessarily have that $\epsilon\ge 0$. If we further substitute $y$ for $\gamma z$, make the change of variable $\bar z=\sqrt{\epsilon} z, \bar v=\sqrt{\epsilon} v$, and redenote the variables $\bar z, \bar v$ by $z,v$, respectively, system \eqref{ddotxy-ne} reduces to the family of first-order systems
\begin{equation}
\label{Lagr-ne}
\begin{cases}
\dot z=u\cr
\dot u=\bigg[\frac{2h}{3m}-\frac{2m(5-9z^4)}{\sqrt{3}(z^2-1)^{1/2}(3z^2+1)^{3/2}}\bigg]z.\cr
\end{cases}
\end{equation}

Standard results of the theory of ordinary differential equations can now be applied to system \eqref{Lagr-ne} to prove the existence and uniqueness of analytic negative elliptic Lagrangian rotopulsators, for admissible initial conditions. To show that, except for a negligible set, they are all rotopulsators, we identify the relative equilibria, which are fixed points of system \eqref{Lagr-ne}.

One fixed point of the vector field in \eqref{Lagr-ne} is obviously $(z,u)=(0,0)$, but it lies outside the domain $z\ge 1$. The other fixed points, if any, must be of the form $(z,0)$, where the positive values of $z$ are given by the roots of the polynomial
$$
Q(z)=27(h^2-9m^4)z^4-18(h^2-15m^4)z^2-8h^2z-h^2-75m^4.
$$
By Descartes's rule of signs, we must distinguish between two cases:

(i) $|h|<\sqrt{15}m^2$, when $Q$ has no positive roots at all;

(ii) $|h|\ge \sqrt{15}m^2$, when $Q$ has exactly one positive root, which is larger than 1 because $Q(1)=-78m^4$ and $Q(z)\to\infty$ when $z\to\infty$. 

Therefore the set of relative equilibria, when the parameters $m, h$ vary, is negligible, so all the other solutions are negative elliptic Lagrangian rotopulsators. This remark completes the proof.
\end{proof}


\section{Negative hyperbolic rotopulsators}

In this section we analyze the solutions given in Definition \ref{def-negative-hyperbolic}. We first prove a criterion for finding such orbits, then provide the conservation laws, and finally analyze two particular classes of examples. First we prove the existence of negative hyperbolic Eulerian rotopulsators of the 2-body problem in $\mathbb H^3$, i.e.\ orbits for which the bodies move on geodesic that rotates hyperbolically. Then we show that for the 3-body problem in $\mathbb H^3$, all Eulerian orbits are relative equilibria, so there are no rotopulsators of this type.

\subsection{Criterion for negative hyperbolic rotopulsators or relative equilibria}

The following result provides necessary and sufficient conditions for the existence of negative hyperbolic rotopulsators or relative equilibria in $\mathbb H^3$.

\begin{criterion}\label{nh-existence}
A solution candidate of the form \eqref{negative-hyperbolic} is a negative hyperbolic rotopulsator for system \eqref{negative} if and only if 
\begin{equation}\label{bet-nh}
\dot\beta=\frac{c}{\sum_{j=1}^Nm_j\rho_j^2}, 
\end{equation}
where $c\ne 0$ is a constant, there are at least two distinct indices $i,j\in\{1,2,\dots,N\}$ such that $q_{ij}$ is not constant, and the variables $w_i, x_i, \rho_i,  i=1,2,\dots, N,$ satisfy the first-order system of\ \! $5N$ equations (with\ \! $N$ constraints: $w_i^2+x_i^2=\rho_i^2-1$, $i=1,2,\dots,N$),
\begin{equation}
\begin{cases}
\label{crit-nh}
\dot w_i=p_i\cr
\dot x_i=s_i\cr
\dot{p}_i=\sum_{\stackrel{j=1}{j\ne i}}^N\frac{m_j(w_j+q_{ij}w_i)}{(q_{ij}^2-1)^{3/2}}+K_i({\bf w},{\bf x},p_i,s_i)w_i\cr
\dot{s}_i=\sum_{\stackrel{j=1}{j\ne i}}^N\frac{m_j(x_j+q_{ij}x_i)}{(q_{ij}^2-1)^{3/2}}+K_i({\bf w},{\bf x},p_i,s_i)x_i,\cr
\rho_i\ddot\beta+2\dot{\rho}_i\dot\beta=\sum_{\stackrel{j=1}{j\ne i}}^N
\frac{m_j\rho_j\sinh(b_j-b_i)}{(q_{ij}^2-1)^{3/2}}, 
\end{cases}
\end{equation}
where ${\bf w}=(w_1,w_2,\dots, w_N), {\bf x}=(x_1, x_2,\dots, x_N)$,
\begin{equation}
K_i({\bf w},{\bf x},p_i,s_i):=\frac{(w_is_i-x_ip_i)^2+p_i^2+s_i^2}{w_i^2+x_i^2+1}+\frac{c^2(w_i^2+x_i^2+1)}{[\sum_{j=1}^Nm_j(w_j^2+x_j^2+1)]^2},
\end{equation}
$i=1,2,\dots, N,$ and, for any $i,j\in\{1,2,\dots,N\}$,
\begin{equation}
\label{nuij}
q_{ij}=w_iw_j+x_ix_j-(w_i^2+x_i^2+1)^{\frac{1}{2}}(w_j^2+x_j^2+1)^{\frac{1}{2}}\cosh(b_i-b_j).
\end{equation}
If the quantities $q_{ij}$ are constant for all distinct indices $i,j\in\{1,2,\dots,N\}$, then the solution is a relative equilibrium. If $q_{ij}=\pm 1$ for some distinct $i,j\in\{1,2,\dots,N\}$, then such solutions don't exist.
\end{criterion}
\begin{proof}
Consider a solution candidate of the form \eqref{negative-hyperbolic} subject to the above initial conditions. Then, for any $i,j\in\{1,2,\dots,N\}$, we obtain the above expression for $q_{ij}$
and, for any $i=1,2,\dots, N$, we find that 
$$
\dot q_{ij}=\frac{(w_i\dot{x}_i-x_i\dot{w}_i)^2+\dot{w}_i^2+\dot{x}_i^2+(w_i^2+x_i^2+1)^2\dot\beta^2}{w_i^2+x_i^2+1}.
$$
For all $\ i=1,2,\dots, N$, each $\rho_i$ can be expressed in terms of $y_i$ and $z_i$ to obtain
$$
\rho_i=(w_i^2+x_i^2+1)^{\frac{1}{2}},\ \ \
\dot{\rho}_i=\frac{w_i\dot{w}_i+x_i\dot{x}_i}{(w_i^2+x_i^2+1)^{\frac{1}{2}}},
$$
$$
\ddot{\rho}_i=\frac{(w_i^2+x_i^2+1)(w_i\ddot{w}_i+x_i\ddot{x}_i)+\dot{w}_i^2+\dot{x}_i^2+(w_i\dot{x}_i-x_i\dot{w}_i)^2}{(w_i^2+x_i^2+1)^{\frac{3}{2}}}.
$$
Substituting a solution of the form \eqref{negative-hyperbolic} into system \eqref{negative} and employing the above formulas, we obtain for the equations corresponding to $\ddot{w}_i$ and $\ddot{x}_i$ that
\begin{equation}
\label{neg-hyp-wi}
\ddot{w}_i=\sum_{\stackrel{j=1}{j\ne i}}^N\frac{m_j(w_j+q_{ij}w_i)}{(q_{ij}^2-1)^{\frac{3}{2}}}+\frac{[(w_i\dot{x}_i-x_i\dot{w}_i)^2+\dot{w}_i^2+\dot{x}_i^2]w_i}{w_i^2+x_i^2+1}+(w_i^2+x_i^2+1)w_i\dot\beta^2,
\end{equation}
\begin{equation}
\label{neg-hyp-xi}
\ddot{x}_i=\sum_{\stackrel{j=1}{j\ne i}}^N\frac{m_j(x_j+q_{ij}x_i)}{(q_{ij}^2-1)^{\frac{3}{2}}}+\frac{[(w_i\dot{x}_i-x_i\dot{w}_i)^2+\dot{w}_i^2+\dot{x}_i^2]x_i}{w_i^2+x_i^2+1}+(w_i^2+x_i^2+1)x_i\dot\beta^2,
\end{equation}
whereas for the equations corresponding to $\ddot{y}_i$ and $\ddot{z}_i$, after some long computations that also use \eqref{neg-hyp-wi} and \eqref{neg-hyp-xi}, we are led to the equations
\begin{equation}
\label{beta-nh}
\rho_i\ddot\beta+2\dot{\rho}_i\dot\beta=\sum_{\stackrel{j=1}{j\ne i}}^N
\frac{m_j\rho_j\sinh(b_j-b_i)}{(q_{ij}^2-1)^{\frac{3}{2}}}, \ \ i=1,2,\dots, N.
\end{equation} 
As in previous criteria, we can conclude that
$$
\dot\beta=\frac{c}{\sum_{i=1}^Nm_i\rho_i^2}=\frac{c}{\sum_{i=1}^Nm_i(w_i^2+x_i^2+1)},
$$
where $c$ is an integration constant. Consequently,
equations \eqref{neg-hyp-wi} and \eqref{neg-hyp-xi} become
\begin{equation}
\label{neg-hyp-wi-new}
\ddot{w}_i=\sum_{\stackrel{j=1}{j\ne i}}^N\frac{m_j(w_j+\nu_{ij}w_i)}{(\nu_{ij}^2-1)^{\frac{3}{2}}}+\frac{[(w_i\dot{x}_i-x_i\dot{w}_i)^2+\dot{w}_i^2+\dot{x}_i^2]w_i}{w_i^2+x_i^2+1}+\frac{a^2(w_i^2+x_i^2+1)w_i}{[\sum_{j=1}^Nm_j(w_j^2+x_j^2+1)]^2},
\end{equation}
\begin{equation}
\label{neg-hyp-xi-new}
\ddot{x}_i=\sum_{\stackrel{j=1}{j\ne i}}^N\frac{m_j(x_j+\nu_{ij}x_i)}{(\nu_{ij}^2-1)^{\frac{3}{2}}}+\frac{[(w_i\dot{x}_i-x_i\dot{w}_i)^2+\dot{w}_i^2+\dot{x}_i^2]x_i}{w_i^2+x_i^2+1}+\frac{a^2(w_i^2+x_i^2+1)x_i}{[\sum_{j=1}^Nm_j(w_j^2+x_j^2+1)]^2}, 
\end{equation}
$i=1,2,\dots, N$. A straightforward computations shows now that equations \eqref{beta-nh}, \eqref{neg-hyp-wi-new}, and \eqref{neg-hyp-xi-new} lead to system \eqref{crit-nh}. The part of the criterion related to relative equilibria follows directly from Definition \ref{def-negative-hyperbolic}. The nonexistence of such solutions if some $q_{ij}=\pm 1$ follows from the fact that at least a denominator cancels in the equations of motion. This remark completes the proof.
\end{proof}

\subsection{Conservation laws for negative hyperbolic rotopulsators}

In addition to Criterion \ref{nh-existence}, we would also like to obtain the conservation laws specific to negative elliptic rotopulsators. They follow by straightforward computations using the above proof, the integral of energy, and the six integrals of the total angular momentum.

\begin{proposition}\label{integrals-nh}
If system \eqref{negative} has a solution of the form \eqref{negative-hyperbolic}, then
the following expressions are constant:

--- energy,
\begin{equation}
\label{energy-nh}
\begin{split}
h=\sum_{i=1}^N\frac{m_i[(w_i\dot{x}_i-x_i\dot{w}_i)^2+\dot{w}_i^2+\dot{x}_i^2]}{2(w_i^2+x_i^2+1)}\hspace{1.5cm}\\+\frac{a^2}{2\sum_{j=1}^Nm_j(w_j^2+x_j^2+1)}+\sum_{1\le i<j\le N}
\frac{m_im_jq_{ij}}{(q_{ij}^2-1)^{1/2}};
\end{split}
\end{equation}

--- total angular momentum relative to the $wx$-plane,
\begin{equation}
\label{angmom-wx-nh}
c_{wx}=0;
\end{equation}

--- total angular momentum relative to the $wy$-plane,
\begin{equation}
\label{angmom-wy-nh}
\begin{split}
c_{wy}=\sum_{i=1}^N\frac{m_i[x_i(w_i\dot x_i-x_i\dot w_i)-\dot w_i]}{(w_i^2+x_i^2+1)^{\frac{1}{2}}}\sinh(\beta+b_i)\hspace{1.2cm}\\
+\frac{a}{\sum_{j=1}^Nm_j(w_j^2+x_j^2+1)}\sum_{i=1}^Nm_iw_i(w_i^2+x_i^2+1)^{\frac{1}{2}}\cosh(\beta+b_i);
\end{split}
\end{equation}

--- total angular momentum relative to the $wz$-plane,
\begin{equation}
\label{angmom-wz-nh}
\begin{split}
c_{wz}=\sum_{i=1}^N\frac{m_i[x_i(w_i\dot x_i-x_i\dot w_i)-\dot w_i]}{(w_i^2+x_i^2+1)^{\frac{1}{2}}}\cosh(\beta+b_i)\hspace{1.2cm}\\
+\frac{a}{\sum_{j=1}^Nm_j(w_j^2+x_j^2+1)}\sum_{i=1}^Nm_iw_i(w_i^2+x_i^2+1)^{\frac{1}{2}}\sinh(\beta+b_i);
\end{split}
\end{equation}

--- total angular momentum relative to the $xy$-plane,
\begin{equation}
\label{angmom-xy-nh}
\begin{split}
c_{xy}=\sum_{i=1}^N\frac{m_i[w_i(\dot w_ix_i-w_i\dot x_i)-\dot x_i]}{(w_i^2+x_i^2+1)^{\frac{1}{2}}}\sinh(\beta+b_i)\hspace{1.2cm}\\
+\frac{a}{\sum_{j=1}^Nm_j(w_j^2+x_j^2+1)}\sum_{i=1}^Nm_ix_i(w_i^2+x_i^2+1)^{\frac{1}{2}}\cosh(\beta+b_i);
\end{split}
\end{equation}

--- total angular momentum relative to the $xz$-plane,
\begin{equation}
\label{angmom-xz-nh}
\begin{split}
c_{xz}=\sum_{i=1}^N\frac{m_i[w_i(\dot w_ix_i-w_i\dot x_i)-\dot x_i]}{(w_i^2+x_i^2+1)^{\frac{1}{2}}}\cosh(\beta+b_i)\hspace{1.2cm}\\
+\frac{a}{\sum_{j=1}^Nm_j(w_j^2+x_j^2+1)}\sum_{i=1}^Nm_ix_i(w_i^2+x_i^2+1)^{\frac{1}{2}}\sinh(\beta+b_i).
\end{split}
\end{equation}
\end{proposition}

--- total angular momentum relative to the $yz$-plane,
\begin{equation}
c_{yz}=-c,
\end{equation}
where $c$ is the constant in the expression of $\dot\beta$ in \eqref{bet-nh}.

\subsection{Negative hyperbolic Eulerian rotopulsators for $N=2$}

Our first example is that of a class of negative hyperbolic rotopulsators for the 2-body problem in $\mathbb H^3$. We call them Eulerian since the bodies are for all time on a geodesic that rotates hyperbolically. 

Consider two bodies of masses $m_1=m_2=:m>0$ and a candidate solution of the form
\begin{equation}\label{sol-2b-nh}
\begin{gathered}
{\bf q}=({\bf q}_1,{\bf q}_2), \ \ {\bf q}_i=(w_i,x_i,y_i,z_i), \ \ i=1,2,\\
w_1=w(t),\ \ x_1=x(t), \ \ y_1=\rho(t)\sinh\beta(t),\ \ z_1=\rho(t)\cosh\beta(t),\\
w_2=-w(t),\ \ x_2=-x(t), \ \ y_2=\rho(t)\sinh\beta(t),\ \ z_2=\rho(t)\cosh\beta(t),\\
\end{gathered}
\end{equation}
with $\beta$ a nonconstant function and $w^2+x^2-\rho^2=-1$. With the help of Criterion \ref{nh-existence}, we can now show that, in general, these are indeed solutions of system \eqref{negative} for $N=2$.

\begin{proposition}
Consider the curved $2$-body problem in $\mathbb H^3$ given by system \eqref{negative} with $N=2$. Then, except for a negligible set of negative hyperbolic Eulerian relative equilibria, every candidate solution of the form \eqref{sol-2b-nh} is a negative hyperbolic Eulerian rotopulsator, which rotates relative to the plane $yz$, but has no rotations with respect to the planes $wx, wy, wz, xy,$ and $xz$.
\end{proposition}
\begin{proof}
Consider a candidate solution of the form \eqref{sol-2b-nh}. Then from Criterion \ref{nh-existence} we can conclude that
$$
q_{12}=-2w^2-2x^2-1,\ \dot\beta=\frac{c}{2m(w^2+x^2+1)^2},
$$
and that a candidate solution of the above form leads to the system
\begin{equation}\label{2-body-nh-eq}
\begin{cases}
\dot w = p\cr
\dot x = s\cr
\dot p = Z^*(w,x,p,s)w\cr
\dot s = Z^*(w,x,p,s)x\cr
\rho\ddot\beta+2\dot\rho\dot\beta=0,
\end{cases}
\end{equation}
where
\begin{equation*}
\begin{gathered}
Z^*(w,x,p,s)=\frac{(ws-xp)^2+p^2+s^2}{w^2+x^2+1}+\frac{c^2}{4m^2(w^2+x^2+1)}\\-\frac{m}{4(w^2+x^2)(w^2+x^2+1)^{1/2}}.
\end{gathered}
\end{equation*}
The third and fourth equations lead to the
conclusion that $w=\zeta x$, where $\zeta$ is a constant. 
Since, by Proposition \ref{integrals-nh}, $c_{wx}=0$, it follows that $ws-xp=0$. If we substitute $w$ for $\zeta x$ and make the change of variables $\bar x=\delta x, \bar s =\delta s$, where $\delta=\zeta^2+1$, and redenote the variables $\bar w,\bar x$ by $w,x$, respectively, system \eqref{2-body-nh-eq} reduces to
\begin{equation}\label{2b-ne-simple}
\begin{cases}
\dot x = s\cr
\dot s = Z(x,s)x,
\end{cases}
\end{equation}
where
$$
Z(x,s) = \frac{s^2}{x^2+1}+\frac{c^2}{4m^2(x^2+1)}-\frac{m}{4x^2(x^2+1)^{1/2}}.
$$
This reduction can be done because, in agreement with the expression given for $\dot\beta$ in Criterion 4, the last equation in \eqref{2-body-nh-eq} can be solved and the constants chosen such that 
$$
\dot\beta = \frac{c}{2m(x^2+1)}.
$$
For given $m>0, c\ne 0$, and  $x>0$, system \eqref{2b-ne-simple} has the single fixed
point
$$
(x,s)=\bigg(\bigg[\frac{m^6+m^3(m^6+4c^4)^{1/2}}{\sqrt{2}c^2}\bigg]^{1/2},0 \bigg),
$$
which produces a relative equilibrium, so when $m$ and $c$ vary the set of relative equilibria is negligible. Standard existence and uniqueness results of the theory of differential equations lead now to the desired conclusion.
It also follows from Proposition \ref{integrals-nh} that 
$c_{yz}\ne 0$, whereas $c_{wy}=c_{wz}=c_{xy}=c_{xz}=0$,
which implies that the binary system rotates relative to the plane $yz$, but has no rotation with respect to the other base planes.
\end{proof}

\subsection{Negative hyperbolic Eulerian relative equilibria for $N=3$}

We further provide a class of specific negative hyperbolic relative equilibria of the curved 3-body problem that follow from Criterion 4, namely Eulerian orbits in $\mathbb H^3$. These systems rotate relative to the plane $yz$, but have no rotations with respect to the planes $wx, wy, wz, xy$, and $xz$.

Consider three equal masses, $m_1=m_2=m_3=:m$, and a candidate solution of the form
\begin{equation}
\label{Eulerian-nh}
\begin{gathered}
{\bf q}=({\bf q}_1,{\bf q}_2,{\bf q}_3), \ \ {\bf q}_i=(w_i,x_i,y_i,z_i), \ \ i=1,2,3,\\
w_1=0,\ \ x_1=0, \ \ y_1=\sinh\beta(t),\ \ z_1=\cosh\beta(t),\\
w_2=w(t),\ \ x_2=x(t), \ \ y_2=\rho(t)\sinh\beta(t),\ \ z_2=\rho(t)\cosh\beta(t),\\
w_3=-w(t),\ \ x_3=-x(t), \ \ y_3=\rho(t)\sinh\beta(t),\ \ z_3=\rho(t)\cosh\beta(t),
\end{gathered}
\end{equation}
with $\beta$ a nonconstant function and $w^2+x^2-\rho^2=-1$.
With the help of Criterion \ref{nh-existence}, we can now show that these are always negative hyperbolic Eulearian relative equilibria of system \eqref{negative} for $N=3$.

\begin{proposition}
Consider the curved $3$-body problem in $\mathbb H^3$ given by system \eqref{negative} for $N=3$. Then every candidate solution of the form \eqref{Eulerian-nh} is a negative hyperbolic Eulerian relative equilibrium, which rotates relative to the plane $yz$, but has no rotations with respect to the planes $wx, wy, wz, xy$, and $xz$.
\end{proposition}
\begin{proof}
Consider a candidate solution of the form \eqref{Eulerian-nh}. Then, using Criterion \ref{nh-existence}, straightforward computations show that
$$
q_{12}=q_{13}=(w^2+x^2+1)^{\frac{1}{2}}=\rho,\ \ q_{23}=-2(w^2+x^2+1)=-2\rho^2,$$
$$
\dot\beta=\frac{a}{m(2w^2+2x^2+3)}=\frac{a}{m(2\rho^2+1)},
$$
and the equations of motion reduce to the system
\begin{equation}
\label{syst-nh}
\begin{cases}
\dot w = p\cr
\dot x = s\cr
\dot p=J^*(w,x,p,s)w\cr
\dot s=J^*(w,x,p,s)x\cr
\ddot\beta=0\cr
\dot\rho\dot\beta=0,
\end{cases}
\end{equation}
where
$$
J^*(w,x,p,s)=\frac{m(w^2+x^2+1)^{1/2}}{(w^2+x^2-1)^{3/2}}-
\frac{m}{(2w^2+2x^2+1)^{3/2}(2w^2+2x^2+3)^{1/2}}
$$
$$
+\frac{(ws-xp)^2+p^2+s^2}{w^2+x^2+1}+\frac{a^2(w^2+x^2+1)}{m^2(2w^2+2x^2+3)^2}.
$$
Both of the last two equations in system \eqref{syst-nh} imply that $\rho$ is constant, which means that any solution, if it exists, must be a relative equilibrium.

As in the previous example, we can conclude that
$w=\zeta x$, where $\zeta$ is a constant, and that 
$ws-xp=0$ since $c_{wx}=0$. Moreover, $c_{wy}=c_{wz}=c_{xy}=c_{xz}=0$, so the orbit has no rotation relative to the planes $wy, wz, xy$, and $xz$.

If we substitute $w$ for $\zeta x$ and make the change of variables $\bar x=\delta x, \bar s =\delta s$, where $\delta=\zeta^2+1$, and redenote the variables $\bar w,\bar x$ by $w,x$, respectively, system \eqref{syst-nh} reduces to
\begin{equation}\label{sys-E-nh}
\begin{cases}
\dot x = s\cr
\dot s = J(x,s)x,
\end{cases}
\end{equation}
where
$$
J(x,s)=\frac{m(x^2+1)^{1/2}}{(x^2-1)^{1/2}}
-\frac{m}{(2x^2+1)^{3/2}(2x^2+3)^{1/2}}+\frac{s^2}{x^2+1}+\frac{a^2(x^2+1)}{m^2(2x^2+3)^2}.
$$

Standard results of the theory of differential equations can now be applied to system \eqref{sys-E-nh} to prove the existence and uniqueness of analytic negative hyperbolic Eulerian relative equilibria, for admissible initial conditions.
Proposition \ref{integrals-nh} shows that $c_{yz}=-c\ne 0$, whereas the computations lead to the conclusion that $c_{wy}=c_{wz}=c_{xy}=c_{xz}=0$, which implies that the system rotates relative to the plane $yz$, but has no rotation with respect to the planes $wx, wy, wz, xy$, and $xz$. This remark completes the proof.
 \end{proof}

\section{Negative elliptic-hyperbolic rotopulsators}

In this section we analyze the solutions given in Definition \ref{def-negative-elliptic-hyperbolic}. We first prove a criterion for finding such orbits, then provide the conservation laws, and finally show that two particular classes of candidate rotopulsators, namely the negative elliptic-hyperbolic Eulerian orbits of the 2- and 3-body problem in $\mathbb H^3$, are entirely formed by relative equilibria, so no rotopulsators of these types occur. 

\subsection{Criterion for negative elliptic-hyperbolic rotopulsators or relative equilibria}

The following result provides necessary and sufficient conditions for the existence of negative elliptic-hyperbolic rotopulsators or relative equilibria in $\mathbb H^3$.

\begin{criterion}\label{neh-existence}
A solution candidate of the form \eqref{negative-elliptic-hyperbolic} is a negative elliptic-hyperbolic rotopulsator for system \eqref{negative} if and only if
\begin{equation}
\label{alpha-and-beta-neh}
\dot\alpha=\frac{d_1}{\sum_{i=1}^Nm_ir_i^2}, \ \ \ \dot\beta=\frac{d_2}{M+\sum_{i=1}^Nm_ir_i^2},
\end{equation}
with $d_1,d_2\ne 0$ constants, there are at least two distinct indices $i,j\in\{1,2,\dots,N\}$ such that $q_{ij}$ is not constant, and the variables $r_1, r_2,\dots, r_N$ satisfy the first-order system of\ \! $4N$ equations,
\begin{equation}
\label{r-neh-new-new}
\begin{cases}
\dot r_i=s_i\cr
\dot{s}_i=L_i({\bf r}, s_i)\cr
r_i\ddot\alpha+2\dot{r}_i\dot\alpha=\sum_{\stackrel{j=1}{j\ne i}}^N
\frac{m_jr_j\sin(a_j-a_i)}{(q_{ij}^2-1)^{3/2}}\cr
\rho_i\ddot\beta+2\dot{\rho}_i\dot\beta=\sum_{\stackrel{j=1}{j\ne i}}^N
\frac{m_j\rho_j\sinh(b_j-b_i)}{(q_{ij}^2-1)^{3/2}},
\end{cases}
\end{equation}
where $r_i^2-\rho_i^2=-1,\ i=1,2,\dots,N$, ${\bf r}=(r_1,r_2,\dots, r_N)$,
\begin{equation}
\begin{split}
L_i({\bf r}, s_i)=r_i(1+r_i^2)\bigg[\frac{d_1^2}{(\sum_{i=1}^Nm_ir_i^2)^2}-\frac{d_2^2}{(M+\sum_{i=1}^Nm_ir_i^2)^2}\bigg]+\frac{r_is_i^2}{1+r_i^2}\hspace{0.3cm}\\
+\sum_{\stackrel{j=1}{j\ne i}}^N\frac{m_j[r_j(1+r_i^2)\cos(a_i-a_j)-r_i(1+r_i^2)^{1/2}(1+r_j^2)^{1/2}\cosh(b_i-b_j)]}{(q_{ij}^2-1)^{3/2}},
\end{split}
\end{equation}
and  for any $i,j\in\{1,2,\dots,N\}$ with $i\ne j$,
$$
q_{ij}=r_ir_j\cos(a_i-a_j)-(1+r_i^2)^{1/2}(1+r_j^2)^{1/2}\cosh(b_i-b_j).
$$
If the quantities $q_{ij}$ are constant for all distinct indices $i,j\in\{1,2,\dots,N\}$, then the solution is a relative equilibrium. If $q_{ij}=\pm 1$ for some distinct $i,j\in\{1,2,\dots,N\}$, then such solutions don't exist.
\end{criterion}
\begin{proof}
Consider a candidate solution of the form \eqref{negative-elliptic-hyperbolic} for system \eqref{negative}. By expressing each $\rho_i$ in terms of $r_i$, $i=1,2,\dots,N$, we obtain that
$$
\rho_i=(1+r_i^2)^{1/2},\ \ \ \dot\rho_i=\frac{r_i\dot r_i}{(1+r_i^2)^{1/2}},\ \ \
\ddot\rho_i=\frac{\dot{r}_i^2+r_i(1+r_i^2)\ddot{r}_i}{(1+r_i^2)^{3/2}},
$$
$q_{ij}$ takes the above form, and
$$
\dot q_{ij}=\dot{r}_i^2+r_i^2\dot\alpha^2-\frac{r_i^2\dot{r}_i^2}{1+r_i^2}-(1+r_i^2)\dot\beta^2.
$$

Substituting a solution candidate of the form \eqref{negative-elliptic-hyperbolic} into system \eqref{negative} and using the above formulas, we obtain for the equations corresponding to $\ddot{w}_i$ and  $\ddot{x}_i$ the equations 
\begin{equation}
\label{r-neh}
\begin{split}
\ddot{r}_i=r_i(1+r_i^2)(\dot\alpha^2-\dot\beta^2)+\frac{r_i\dot{r}_i^2}{1-r_i^2}\hspace{3.5cm}\\
+\sum_{\stackrel{j=1}{j\ne i}}^N\frac{m_j[r_j(1+r_i^2)\cos(a_i-a_j)-r_i(1+r_i^2)^{1/2}(1+r_j^2)^{1/2}\cosh(b_i-b_j)]}{(q_{ij}^2-1)^{3/2}},
\end{split}
\end{equation}
\begin{equation}
\label{alpha-neh}
r_i\ddot\alpha+2\dot{r}_i\dot\alpha=\sum_{\stackrel{j=1}{j\ne i}}^N
\frac{m_jr_j\sin(a_j-a_i)}{(q_{ij}^2-1)^{3/2}}, \ \ i=1,2,\dots, N,
\end{equation} 
whereas for the equations corresponding to $\ddot{y}_i, \ddot{z}_i$, we find equations \eqref{r-neh} again as well as the equations
\begin{equation}
\label{beta-neh}
\rho_i\ddot\beta+2\dot{\rho}_i\dot\beta=\sum_{\stackrel{j=1}{j\ne i}}^N
\frac{m_j\rho_j\sinh(b_j-b_i)}{(q_{ij}^2-1)^{3/2}}, \ \ i=1,2,\dots, N.
\end{equation}
As in Criterion 2, we obtain
$$
\dot\alpha=\frac{d_1}{\sum_{i=1}^Nm_ir_i^2},\ \ \dot\beta=\frac{d_2}{M+\sum_{i=1}^Nm_ir_i^2},
$$
where $d_1, d_2$ are integration constants, so \eqref{r-neh-new-new} is a first-order system of\ \! $4N$ equations.
The part of the statement related to relative equilibria follows directly from Definition \ref{def-negative-elliptic-hyperbolic}. The nonexistence of such solutions if some $q_{ij}=\pm 1$ follows from the fact that at least a denominator cancels in the equations of motion. This remark  completes the proof.
\end{proof}

\begin{remark}
From \eqref{alpha-and-beta-neh}, we can conclude that $\dot\alpha$ and $\dot\beta$
are not independent of each other, but connected by the relationship
\begin{equation}\label{alpha+beta-neh}
\frac{d_2}{\dot\beta}-\frac{d_1}{\dot\alpha}=M.
\end{equation}
\end{remark}
 
\begin{remark}
Criteria 2 and 5 involve first-order systems of\ \! $4N$
equations, whereas Criteria 1, 3, and 4 involve first-order systems 
of\ \! $5N$ equations with $N$ constraints, so the dimension of the systems is the same in all cases.
\end{remark}

\subsection{Conservation laws for negative elliptic-hyperbolic rotopulsators}

In addition to Criterion \ref{neh-existence}, we would also like to obtain the conservation laws specific to negative elliptic rotopulsators. These laws follow by straightforward computations using the above proof, the integral of energy, and the six integrals of the total angular momentum.

\begin{proposition}\label{integrals-neh}
If system \eqref{negative} has a solution of the form \eqref{negative-elliptic-hyperbolic}, then the following expressions are constant:

--- energy,
\begin{equation}
\label{energy-neh}
h=\sum_{i=1}^N\frac{m_i\dot r_i^2}{2(1+r_i^2)}+\frac{d_1^2}{2\sum_{j=1}^Nm_jr_j^2}+\frac{d_2^2}{2(M+\sum_{j=1}^Nm_jr_j^2)}+\sum_{1\le i<j\le N}\frac{m_im_jq_{ij}}{(q_{ij}^2-1)^{\frac{1}{2}}};
\end{equation}

--- total angular momentum relative to the $wx$-plane,
\begin{equation}
c_{wx}=d_1,
\end{equation}
where $d_1\ne 0$ is the constant in the expression of $\dot\alpha$ in \eqref{alpha-and-beta-neh}.

--- total angular momentum relative to the $wy$-plane,
\begin{equation}
\label{angmom-wy-neh}
\begin{split}
c_{wy}=-\sum_{i=1}^N\frac{m_i\dot r_i}{(1+r_i^2)^{\frac{1}{2}}}\cos(\alpha+a_i)\sinh(\beta+b_i)\hspace{2.7cm}\\ 
+\sum_{i=1}^Nm_ir_i(1+r_i^2)^{\frac{1}{2}}[\dot\alpha\sin(\alpha+a_i)\sinh(\beta+b_i)+\dot\beta\cos(\alpha+a_i)\cosh(\beta+b_i)];
\end{split}
\end{equation}

--- total angular momentum relative to the $wz$-plane,
\begin{equation}
\label{angmom-wz-neh}
\begin{split}
c_{wz}=-\sum_{i=1}^N\frac{m_i\dot r_i}{(1+r_i^2)^{\frac{1}{2}}}\cos(\alpha+a_i)\cosh(\beta+b_i)\hspace{2.7cm}\\ 
+\sum_{i=1}^Nm_ir_i(1+r_i^2)^{\frac{1}{2}}[\dot\alpha\sin(\alpha+a_i)\cosh(\beta+b_i)+\dot\beta\cos(\alpha+a_i)\sinh(\beta+b_i)];
\end{split}
\end{equation}

--- total angular momentum relative to the $xy$-plane,
\begin{equation}
\label{angmom-xy-neh}
\begin{split}
c_{xy}=-\sum_{i=1}^N\frac{m_i\dot r_i}{(1+r_i^2)^{\frac{1}{2}}}\sin(\alpha+a_i)\sinh(\beta+b_i)\hspace{2.7cm}\\ 
+\sum_{i=1}^Nm_ir_i(1+r_i^2)^{\frac{1}{2}}[\dot\beta\sin(\alpha+a_i)\cosh(\beta+b_i)-\dot\alpha\cos(\alpha+a_i)\sinh(\beta+b_i)];
\end{split}
\end{equation}

--- total angular momentum relative to the $xz$-plane,
\begin{equation}
\label{angmom-xz-neh}
\begin{split}
c_{xz}=-\sum_{i=1}^N\frac{m_i\dot r_i}{(1+r_i^2)^{\frac{1}{2}}}\sin(\alpha+a_i)\cosh(\beta+b_i)\hspace{2.7cm}\\ 
+\sum_{i=1}^Nm_ir_i(1+r_i^2)^{\frac{1}{2}}[\dot\beta\sin(\alpha+a_i)\sinh(\beta+b_i)-\dot\alpha\cos(\alpha+a_i)\cosh(\beta+b_i)].
\end{split}
\end{equation}
\end{proposition}

--- total angular momentum relative to the $yz$-plane,
\begin{equation}
c_{yz}=-d_2,
\end{equation}
where $d_2\ne 0$ is the constant in the expression of $\dot\beta$ in \eqref{alpha-and-beta-neh}.

\subsection{Negative elliptic-hyperbolic Eulerian rotopulsators for $N=2$}

We further provide a class of negative elliptic-hyperbolic rotopulsators of the curved 2-body problem in $\mathbb H^3$, namely Eulerian orbits, for which the bodies move on a rotating geodesic. These systems rotate relative to the planes $wx$ and $yz$, but have no rotation with respect to the other base planes.

Consider two equal masses, $m_1=m_2=:m>0$, and a candidate solution of the form
\begin{equation}\label{eulerian-2-neh}
\begin{split}
{\bf q}=({\bf q}_1,{\bf q}_2), \ \ {\bf q}_i=(w_i,x_i,y_i,z_i), \ \ i=1,2,\\
w_1=r(t)\cos\alpha(t),\ x_1=r(t)\sin\alpha(t),\hspace{0.5cm}\\
 y_1=\rho(t)\sinh\beta(t),\ z_1=\rho(t)\cosh\beta(t),\hspace{0.3cm}\\
w_2=-r(t)\cos\alpha(t),\ x_2=-r(t)\sin\alpha(t),\hspace{0.15cm}\\
 y_2=\rho(t)\sinh\beta(t),\ z_2=\rho(t)\cosh\beta(t),\hspace{0.3cm}
\end{split}
\end{equation}
with $\alpha$ and $\beta$ nonconstant functions and $r^2-\rho^2=-1$. We can now prove a result which shows that solutions of the form \eqref{eulerian-2-neh} are always relative equlibria, but can never form rotopulsators.

\begin{proposition}
Consider the curved $2$-body problem in $\mathbb H^3$ given by system \eqref{negative} with $N=2$. Then every candidate solution of the form \eqref{eulerian-2-neh} is a negative elliptic-hyperbolic Eulerian rotopulsator, which rotates relative to the planes $wx$ and $yz$, but has
no rotation with respect to the planes $wy, wz, xy$, and $xz$.
\end{proposition}
\begin{proof}
Consider a candidate solution of the form \eqref{eulerian-2-neh}. Then the variables relevant to Criterion 5 take the form
$$
q_{12}=-1-2r^2, \ \ \dot\alpha=\frac{d_1}{2mr^2}, \ \ \dot\beta=\frac{d_2}{2m(1+r^2)}, \ \ {\rm with}\ \ d_1,d_2\ \ {\rm constants}.
$$
Then system \eqref{r-neh-new-new} reduces to
\begin{equation}\label{neh-2}
\begin{cases}
\dot r = s\cr
\dot s = L(r,s)\cr
r\ddot\alpha + 2\dot r\dot\alpha=0\cr
\rho\ddot\beta + 2\dot \rho\dot\beta=0,
\end{cases}
\end{equation}
where
$$
L(r,s) = r(1+r^2)\bigg[\frac{d_1^2}{4m^2r^4}-\frac{d_2^2}{4m^2(1+r^2)^2}\bigg]+\frac{rs^2}{1+r^2}-\frac{m}{4r^2(1+r^2)^2}.
$$
But using the above expressions of $\dot\alpha, \dot\beta$ and the fact that $\rho=(1+r^2)^{1/2}, \dot\rho=\frac{r\dot r}{(1+r^2)^{1/2}}$,
we can see that the last two equations in \eqref{neh-2} are identically satisfied, so system \eqref{neh-2} reduces to its first two equations. Since there are no other constraints on $r$, standard results of the theory of differential equations then prove the existence of these Eulerian rotopulsators for $N=2$, thus showing that elliptic-hyperbolic rotopulsators exist in $\mathbb H^3$.

Using Proposition \ref{integrals-neh}, we can conclude that
$c_{wx}=d_1\ne 0, c_{yz}=-d_2\ne 0$, whereas $c_{wy}=c_{wz}=c_{xy}=c_{xz}=0$, which implies that the binary system rotates relative to the planes $wx$ and $yz$, but has no rotation with respect to the other four base planes. This remark completes the proof.
\end{proof}

\subsection{Negative elliptic-hyperbolic Eulerian relative equilibria for $N=3$}

We further provide a class of negative elliptic-hyperbolic relative equilibria of the curved 3-body problem in $\mathbb H^3$, namely Eulerian orbits. These systems rotate relative to the $wx$ and $yz$ planes, but have no rotation with respect to the planes $wy, wz, xy$, and $xz$. 

Consider three equal masses, $m_1=m_2=m_3:=m$, and a candidate solution of the form
\begin{equation}\label{eulerian-neh}
\begin{split}
{\bf q}=({\bf q}_1,{\bf q}_2,{\bf q}_3), \ \ {\bf q}_i=(w_i,x_i,y_i,z_i), \ \ i=1,2,3,\\
w_1=0,\ x_1=0,\ y_1=\sinh\beta(t),\ z_1=\cosh\beta(t),\hspace{0.1cm}\\
w_2=r(t)\cos\alpha(t),\ x_2=r(t)\sin\alpha(t),\hspace{1.2cm}\\
y_2=\rho(t)\sinh\beta(t),\ z_2=\rho(t)\cosh\beta(t),\hspace{1cm}\\
w_3=-r(t)\cos\alpha(t),\ x_3=-r(t)\sin\alpha(t),\hspace{0.8cm}\\
 y_3=\rho(t)\sinh\beta(t),\ z_3=\rho(t)\cosh\beta(t),\hspace{0.95cm}
\end{split}
\end{equation}
with $\alpha$ and $\beta$ nonconstant functions and $r^2-\rho^2=-1$. 
We can now prove a result which shows that solutions of the form \eqref{eulerian-neh} are always relative equlibria, but can never form rotopulsators.

\begin{proposition}
Consider the curved $3$-body problem in $\mathbb H^3$ given by system \eqref{negative} with $N=3$. Then every candidate solution of the form \eqref{eulerian-neh} is a negative elliptic-hyperbolic Eulerian relative equilibrium, which rotates relative to the planes $wx$ and $yz$, but has no rotation with respect to the planes $wy, wz, xy$, and $xz$.
\end{proposition}
\begin{proof}
Consider a candidate solution of the form \eqref{eulerian-neh} of system \eqref{negative}. Then the variables relevant to Criterion \ref{neh-existence} take the form
$$
q_{12}=q_{13}=-(1+r^2)^{1/2},\ \ q_{23}=-2r^2-1,
$$
\begin{equation}\label{alpha-beta-dot-neh}
\dot\alpha=\frac{d_1}{2mr^2},\ \  \dot\beta=\frac{d_2}{m(3+2r^2)}, \ \ \! {\rm with} \ \ \! d_1, d_2\ne 0\ \ \! {\rm constants}.
\end{equation}
Then system \eqref{r-neh-new-new} reduces to
\begin{equation}\label{neh-RE}
\begin{cases}
\dot r = s\cr
\dot s = R(r,s)\cr
r\ddot\alpha+2\dot r\dot\alpha = 0\cr
\ddot\beta = 0\cr
\dot \rho\dot\beta = 0,
\end{cases}
\end{equation}
where
$$
R(r,s)=r(1+r^2)\bigg[\frac{d_1^2}{4m^2r^4}-\frac{d_2^2}{m^2(3+2r^2)^2}\bigg]+\frac{rs^2}{1+r^2}-\frac{m(5+4r^2)}{4r^2(1+r^2)^{1/2}}.
$$
The third equation in \eqref{neh-RE} is identically satisfied, but the last two equations and the expression of $\dot\beta$ in \eqref{alpha-beta-dot-neh} show that $\rho$ is constant (we take it positive, since $\rho=0$ does not correspond to a valid solution), so $r$ is constant as well. Therefore any solution of this type, if it exists, is a relative equilibrium. Notice that $\alpha$ and $\beta$ don't have to be equal, although they are not independent of each other, as relation \eqref{alpha+beta-neh} shows.

Standard results of the theory of differential equations show now that, for admissible initial conditions, negative elliptic hyperbolic relative equilibria of the 3-body problem, given by system \eqref{neh-RE}, do exist.

Using Proposition \ref{integrals-neh}, we can conclude that
$c_{wx}=d_1\ne 0, c_{yz}=-d_2\ne 0$, whereas $c_{wy}=c_{wz}=c_{xy}=c_{xz}=0$, which implies that the system rotates relative to the plane $yz$, but has no rotation with respect to the other base planes. This remark completes the proof.
\end{proof}

\section{Qualitative behaviour of rotopulsators in $\mathbb H^3$}

In this section we will prove a result that describes the qualitative behaviour of rotopulsators in $\mathbb H^3$.
For this purpose, we first introduce a geometric topological object that plays in $\mathbb H^3$ the role the Clifford torus plays in $\mathbb S^3$, and briefly present some of its properties.


\subsection{Hyperbolic cylinders} The 2-dimensional manifold defined by
\begin{equation}
\label{Cyl}
{\bf C}_{r\rho}:=\{(w,x,y,z)\in\mathbb M^{3,1}\ |\ r^2-\rho^2=-1, \ \ 0\le \theta < 2\pi, \ \xi\in\mathbb R\},
\end{equation}
where $w=r\cos\theta, x=r\sin\theta, y=\rho\sinh\xi$, and
$z=\rho\cosh\xi$, with $r,\rho\ge 0$, is called a hyperbolic cylinder, and it has constant positive curvature for $r$ and $\rho$ fixed. But ${\bf C}_{r\rho}$ also lies in $\mathbb H^3$ because the coordinates $w,x,y,z$ satisfy the equation
$$
w^2+x^2+y^2-z^2=-1.
$$
When $r$ (and, consequently, $\rho$) takes all admissible positive real values, the family of hyperbolic cylinders thus defined foliates $\mathbb H^3$.

We have previously shown that negative relative equilibria rotate on hyperbolic cylinders, \cite{Diacu3}, \cite{Diacu4}. We will next prove that, at every moment in time, at least one body of a rotopulsator of $\mathbb H^3$ passes through a continuum of hyperbolic cylinders of any given foliation of $\mathbb H^3$.

\subsection{Geometry and dynamics of rotopulsators in $\mathbb H^3$}

We can now state and prove the following result.
\begin{theorem}
Consider a negative elliptic, negative hyperbolic, or negative elliptic-hyperbolic rotopulsator of the curved $N$-body problem in $\mathbb H^3$. Then, for any foliation $({\bf C}_{r\rho})_{r,\rho>0}$ of $\mathbb H^3$ given by hyperbolic cylinders, it is impossible that the trajectory of each body is contained for all time in some hyperbolic cylinder. In other words, for any such foliation, there is at least one body whose trajectory intersects a continuum of hyperbolic cylinders.
\end{theorem}
\begin{proof}
Let us assume that there exists a foliation $({{\bf C}_{r\rho}})_{r,\rho\ge 0}$ of $\mathbb H^3$ with hyperbolic cylinders for which a solution of the form \eqref{negative-elliptic}, \eqref{negative-hyperbolic}, or \eqref{negative-elliptic-hyperbolic} behaves such that the trajectory of each body is confined to a hyperbolic cylinder. We will prove that such a solution must be a relative equilibrium.

Indeed, if the body $m_i$, whose solution is described by 
$$
w_i=r_i\cos(\alpha+a_i), x_i=r_i\sin(\alpha+a_i), y_i=\rho_i\sinh(\beta+b_i), z_i=\rho_i\cosh(\beta+b_i),
$$
with $r_i^2-\rho_i^2=-1$, is confined to the hyperbolic cylinder ${\bf C}_{r^\sharp\rho^\sharp}$, with $r^\sharp,\rho^\sharp$ constant, then
$r_i=r^\sharp$ and $\rho_i=\rho^\sharp$ are also constant. Similarly,  if the body $m_j$, whose solution is described by 
$$
w_j=r_i\cos(\alpha+a_j), x_j=r_i\sin(\alpha+a_j), y_i=\rho_j\sinh(\beta+b_j), z_j=\rho_j\cosh(\beta+b_j),
$$
with $r_j^2-\rho_j^2=-1$, is confined to the hyperbolic cylinder ${\bf C}_{r^\flat\rho^\flat}$, with $r^\flat,\rho^\flat$ constant, then
$r_i=r^\flat$ and $\rho_i=\rho^\flat$ are also constant. As a result, 
$$
q_{ij}=r^\sharp r^\flat\cos(a_i-a_j)-\rho^\sharp\rho^\flat\cosh(b_i-b_j),
$$ 
which is constant. So all the mutual distances are constant, therefore the solution is a relative equilibrium. 

For negative elliptic rotopulsators, we can use Remark \ref{def-nonstandard}, and notice that they are negative elliptic-hyperbolic rotopulsators with $\beta\equiv 0$. But 
since $\beta$ does not occur anyway in the above expression of $q_{ij}$, it won't show up for $\beta\equiv 0$ either, so the mutual distances of such orbits are also constant. For negative hyperbolic rotopulsators we can draw the same conclusion by using the fact that they are elliptic-hyperbolic rotopulsators with $\alpha\equiv 0$.

Since the foliation of $\mathbb H^3$ with hyperbolic cylinders forms a continuum of surfaces, the last part of the theorem follows. This remark completes the proof.
\end{proof}

\section{Final remarks}

Criteria 1, 2, 3, 4, and 5 open the possibility to study large classes of rotopulsators for various values of $N\ge 2$. 
In this paper we put into the evidence classes of positive elliptic, positive elliptic-elliptic, negative elliptic, and negative hyperbolic rotopulsators and gave some examples for $N=2$ and $N=3$. Moreover, the above criteria allow us to find classes of relative equilibria that are more difficult to find using the results obtained in \cite{Diacu3} and \cite{Diacu4}.

Since the classes of solution we found in this preliminary paper are restricted to $N=2$, $N=3$ and to configurations having many symmetries, and consequently equal masses, it would be interesting to know whether there exist rotopulsators given for $N>3$ bodies, less symmetric configurations, as well as nonequal masses.  

We derived criteria for positive elliptic, positive elliptic-elliptic, and negative elliptic Lagrangian rotopulsators and for the negative hyperbolic and negative elliptic-hyperbolic Eulerian rotopulsators in the 3-dimensional case, but used them mainly to obtain existence and uniqueness results, although, in some cases, we hinted at the qualitative behaviour of the classes of orbits we found. It would be therefore interesting to study these equations in detail and provide a complete classification of the motions that occur, as it has been done in \cite{Diacu5} for the 2-dimensional case of the curved 3-body problem. Since the equations for these orbits in $\mathbb S^3$ and $\mathbb H^3$ are more complicated than the ones in $\mathbb S^2$ and $\mathbb H^2$, each system would require an extensive study. Finally, the stability of these orbits, using for instance the tools developed in \cite{Martinez} and \cite{Diacu9}, is another topic that merits close attention and points at further directions in which the curved $N$-body problem can be successfully developed.

\bigskip

\noindent{\bf Acknowledgements.} Florin Diacu is indebted to Carles Sim\'o, James Montaldi, and Sergiu Popa for some enlightening discussions. The authors also acknowledge the partial support provided by a Discovery Grant from NSERC of Canada (Florin Diacu) and a University of Victoria Fellowship (Shima Kordlou).

\end{document}